\documentclass[11pt]{amsart}

\usepackage{amssymb}
\usepackage{url}

\usepackage{wasysym}
\usepackage{amsmath, calc, mathdots}
\usepackage{tikz-cd}

\usepackage{mathrsfs,multicol}
\usepackage[pagebackref, hypertexnames=false, colorlinks, citecolor=red, linkcolor=red]{hyperref}

	\normalbaselines						  %
\def\RR{\mathbb{R}}
\def\ZZ{\mathbb{Z}}
\def\CC{\mathbb{C}}
\def\NN{\mathbb{N}}

\newcommand{\al}{{\alpha}}
\newcommand{\la}{{\lambda}}
\newcommand{\f}{{\varphi}}

\newcommand{\cX}{{\mathcal{X}}}

\newcommand{\te}{{\theta}}

\newcommand{\dd}{{d}}

\newcommand{\fdot}{\,\cdot\,}

\makeatletter
\def\Ddots{\mathinner{\mkern1mu\raise\p@
\vbox{\kern7\p@\hbox{.}}\mkern2mu
\raise4\p@\hbox{.}\mkern2mu\raise7\p@\hbox{.}\mkern1mu}}
\makeatother

\newcommand{\cH}{\mathcal{H}}

\newcommand{\cB}{\mathcal{B}}

\newcommand{\cA}{\mathcal{A}}
\newcommand{\eps}{\varepsilon}

\newcommand{\cD}{\mathcal{D}}
\newcommand{\cG}{\mathcal{G}}

\newcommand{\fh}{\mathfrak{h}}
\newcommand{\fK}{\mathfrak{K}}
\newcommand{\cJ}{\mathcal{J}}
\newcommand{\cW}{\mathcal{W}}
\newcommand{\cI}{\mathcal{I}}
\newcommand{\fI}{\mathfrak{I}}

\DeclareMathOperator{\dom}{dom}

\DeclareMathOperator{\spa}{span}

\newcommand{\ci}[1]{_{ {}_{\scriptstyle #1}}}
\newcommand{\ti}[1]{_{\scriptstyle \text{\rm #1}}}
\catcode`@=11
\chardef\mathlig@atcode\count255

\def\actively#1#2{\begingroup\uccode`\~=`#2\relax\uppercase{\endgroup#1~}}
\def\mathlig@gobble{\afterassignment\mathlig@next@cmd\let\mathlig@next= }
\def\mathlig@delim{\mathlig@delim}
\def\mathlig@defcs#1{\expandafter\def\csname#1\endcsname}
\def\mathlig@let@cs#1#2{\expandafter\let\expandafter#1\csname#2\endcsname}
\def\mathlig@appendcs#1#2{\expandafter\edef\csname#1\endcsname{\csname#1\endcsname#2}}

\def\mathlig#1#2{\mathlig@checklig#1\mathlig@end\mathlig@defcs{mathlig@back@#1}{#2}\ignorespaces}

\def\mathlig@checklig#1#2\mathlig@end{%
 \expandafter\ifx\csname mathlig@forw@#1\endcsname\relax
 \expandafter\mathchardef\csname mathlig@back@#1\endcsname=\mathcode`#1%
 \mathcode`#1"8000\actively\def#1{\csname mathlig@look@#1\endcsname}%
 \mathlig@dolig#1\mathlig@delim
\fi
\mathlig@checksuffix#1#2\mathlig@end
}

\def\mathlig@checksuffix#1#2\mathlig@end{%
\ifx\mathlig@delim#2\mathlig@delim\relax\else\mathlig@checksuffix@{#1}#2\mathlig@end\fi
}
\def\mathlig@checksuffix@#1#2#3\mathlig@end{%
\expandafter\ifx\csname mathlig@forw@#1#2\endcsname\relax\mathlig@dosuffix{#1}{#2}\fi
\mathlig@checksuffix{#1#2}#3\mathlig@end
}

\def\mathlig@dosuffix#1#2{%
\mathlig@appendcs{mathlig@toks@#1}{#2}%
\mathlig@dolig{#1}{#2}\mathlig@delim
}

\def\mathlig@dolig#1#2\mathlig@delim{%
 \mathlig@defcs{mathlig@look@#1#2}{%
 \mathlig@let@cs\mathlig@next{mathlig@forw@#1#2}\futurelet\mathlig@next@tok\mathlig@next}%
 \mathlig@defcs{mathlig@forw@#1#2}{%
  \mathlig@let@cs\mathlig@next{mathlig@back@#1#2}%
  \mathlig@let@cs\checker{mathlig@chck@#1#2}%
  \mathlig@let@cs\mathligtoks{mathlig@toks@#1#2}%
  \expandafter\ifx\expandafter\mathlig@delim\mathligtoks\mathlig@delim\relax\else
  \expandafter\checker\mathligtoks\mathlig@delim\fi
  \mathlig@next
 }%
 \mathlig@defcs{mathlig@toks@#1#2}{}%
 \mathlig@defcs{mathlig@chck@#1#2}##1##2\mathlig@delim{%
  \ifx\mathlig@next@tok##1%
   \mathlig@let@cs\mathlig@next@cmd{mathlig@look@#1#2##1}\let\mathlig@next\mathlig@gobble
  \fi
  \ifx\mathlig@delim##2\mathlig@delim\relax\else
   \csname mathlig@chck@#1#2\endcsname##2\mathlig@delim
  \fi
 }%
 \ifx\mathlig@delim#2\mathlig@delim\else
  \mathlig@defcs{mathlig@back@#1#2}{\csname mathlig@back@#1\endcsname #2}%
 \fi
}%

\catcode`@\mathlig@atcode

\mathchardef\ordinarycolon\mathcode`\:
\def\vcentcolon{\mathrel{\mathop\ordinarycolon}}
\mathlig{:=}{\vcentcolon=}
\mathlig{::=}{\vcentcolon\vcentcolon=}

\numberwithin{equation}{section}

\theoremstyle{plain}
\newtheorem{theo}{Theorem}[section]
\newtheorem{cor}[theo]{Corollary}
\newtheorem{lem}[theo]{Lemma}
\newtheorem{prop}[theo]{Proposition}

\theoremstyle{definition}
\newtheorem{defn}[theo]{Definition}

\newtheorem*{theorem*}{Theorem}
\theoremstyle{remark}

\newtheorem*{ex*}{Example}
\theoremstyle{remark}
\newtheorem*{exs*}{Examples}
\theoremstyle{remark}
\newtheorem*{rems*}{Remarks}
\theoremstyle{remark}
\newtheorem*{rem*}{Remark}

\theoremstyle{definition}
\newtheorem*{idea*}{Idea}

\title[Boundary Triples and $m$-functions for Powers of the Jacobi Operator]{Boundary Triples and Weyl $m$-functions for Powers of the Jacobi Differential Operator}

\author{Dale~Frymark}
\address{Department of Mathematics, Stockholm University, Kr\"aftriket 6, 106 91 Stockholm, Sweden.}
\email{dale@math.su.se}

\keywords{Boundary Triples, Self-Adjoint Extension Theory, Singular Sturm--Liouville Operators, Nevanlinna--Herglotz Functions, Weyl $m$-functions}
\thanks{The author was partially supported by the Swedish Foundation for Strategic Research under grant AM13-0011.}
 \subjclass[2010]{34B20, 34B24, 34L15, 47B25, 47E05}

\begin{document}

\begin{abstract}
The abstract theory of boundary triples is applied to the classical Jacobi differential operator and its powers in order to obtain the Weyl $m$-function for several self-adjoint extensions with interesting boundary conditions: separated, periodic and those that yield the Friedrichs extension. These matrix-valued Nevanlinna--Herglotz $m$-functions are, to the best knowledge of the author, the first explicit examples to stem from singular higher-order differential equations.

The creation of the boundary triples involves taking pieces, determined in \cite{FL}, of the principal and non-principal solutions of the differential equation and putting them into the sesquilinear form to yield maps from the maximal domain to the boundary space. These maps act like quasi-derivatives, which are usually not well-defined for all functions in the maximal domain of singular expressions. However, well-defined regularizations of quasi-derivatives are produced by putting the pieces of the non-principal solutions through a modified Gram--Schmidt process.
\end{abstract}

\maketitle

\setcounter{tocdepth}{1}
\tableofcontents

%%%%%%%%%%%%%%%%%%%%%%%%%%%%%
%%%%%%%%%%%%%%%%%%%%%%%%%%%%%
\section{Introduction}
%%%%%%%%%%%%%%%%%%%%%%%%%%%%%
%%%%%%%%%%%%%%%%%%%%%%%%%%%%%

The necessary boundary conditions for self-adjoint extensions of Sturm--Liouville operators with limit-circle endpoints are usually more difficult to determine than in the regular case (i.e.~the 1D Schr\"odinger operator). This is mainly due to the fact that Dirichlet and Neumann boundary conditions no longer yield self-adjoint extensions in these cases. Other tools that come from perturbation theory and describe the spectral theory of changing boundary conditions also appear to have not been implemented yet for limit-circle endpoints \cite{AK, S}. 

Let $\ell[\fdot]$ be a symmetric differential expression on the weighted space $L^2[(a,b),w]$ that is in the limit-circle case at the endpoints $a,b\in\RR\cup\{\pm\infty\}$. This classification implies there exist two linearly independent solutions, say $u(x,\la)$ and $v(x,\la)$, to the equation 
\begin{align*}
    (\ell-\la)f=0,
\end{align*}
at each endpoint, where $f$ is taken from the operator's associated maximal domain $\cD\ti{max}$. The Glazman--Krein--Naimark (GKN) theory says that all domains associated with self-adjoint extensions can be obtained by imposing boundary conditions that use these solutions, or other functions that satisfy the conditions of Theorem \ref{t-gkn2}, in the sesquilinear form.

Boundary conditions in the regular case are often related to quasi-derivatives, which are natural building blocks for the sesquilinear form. These expressions are not well-defined for all $f\in\cD\ti{max}$ in the limit-circle case, and the culprit is the non-principal solution $v(x,\la)$. This solution in the sesquilinear form produces only a regularization of the $0$-th quasi-derivative, which can be seen explicitly e.g.~in the recent manuscript \cite{FLN}. Fortunately, this regularization can still be used to create self-adjoint extensions and determine spectral properties by constructing the Weyl $m$-function. Analysis of boundary conditions and the $m$-function in the limit-circle case, in general and for examples, can be found in \cite{All, B, BH, BG, EK, Ful, FP, GZ, MZ, NZ, Titch}

However, these problems become more pronounced and difficult when higher-order ordinary differential equations are considered. Significant progress has been made in recent years in describing the boundary conditions that yield self-adjoint extensions for such operators \cite{Codd, EI, Kod, ASZ1, ASZ2, ASZ3, AZ, AZBook, WaWe}. Recent developments in left-definite theory \cite{DHS, ELT, LW02, LW13, LW15} mean it is now more convenient to work with powers of Sturm--Liouville operators that are bounded from below. In this manuscript, we focus on powers of the Jacobi differential operator for several reasons: the Jacobi expression is extremely well-studied \cite{B, BEZ, E, I, K, Sz} and contains many interesting examples as special choices of its parameters, it is limit-circle at both endpoints, analytic properties of domains of powers were given in \cite{EKLWY}, and the recent manuscripts \cite{FFL, FL} gave new insight into the structure of the defect spaces of powers.

Boundary triples are naturally applicable to the study of boundary conditions for Sturm--Liouville operators \cite{B07, DHMS, KK, MN} thanks to their connections with the sesquilinear form and quasi-derivatives. The recent book \cite{BdS} gives an extensive treatment of these applications and forms the basis for the background of Section \ref{s-bg}. In particular, the theory of boundary triples yields formulas for the $\gamma$-field, Weyl $m$-function and the transformation of boundary conditions.

Indeed, we obtain the Weyl $m$-function explicitly for several important self-adjoint extensions of the Jacobi differential operator and its powers. Even in the uncomposed case, this is thought to be new despite how studied the operator is. Examples of Weyl $m$-functions for other classical, but surprisingly unknown, examples can be found in \cite{FLN}. The analysis of the $n$-th power of the Jacobi operator yields a Weyl $m$-function that is a $(2n\times2n)$ matrix-valued Nevanlinna--Herglotz function, see \cite{GT} for more on these functions. Matrix-valued Nevanlinna--Herglotz functions have been obtained when using operator-valued potentials of Schr\"odinger operators, e.g.~\cite{GWZ}, and matrix-valued Sturm--Liouville operators, e.g.~\cite{CG, CGN}, but to the best knowledge of the author this is the first time one has been explicitly determined when singular endpoints are present.

The method of boundary triples is only able to be applied after obtaining operations that act as quasi-derivatives on $\cD\ti{max}$. The linear span of pieces of principal and non-principal solutions were found to constitute a basis for the defect spaces in \cite{FL}, implying that they should generate such operations when put into the sesquilinear form. This is partly true, as the $2n$ pieces of the principal solution generate exactly the $n$-th through $2n-1$-st quasi-derivatives at each endpoint, see Definition \ref{d-quasiderivgeneral}. However, the first $n$ quasi-derivatives prove elusive; mimicking the main obstacle in the uncomposed Sturm--Liouville case. See i.e.~\cite{EI} for more on how quasi-derivatives are central to the self-adjoint extension theory of higher-order ordinary differential equations.

A matrix of sesquilinear forms that shows the interaction between pieces of solutions reveals that the generated operations are degenerate in some sense. The pieces of the non-principal solutions are then put through a modified Gram--Schmidt process, associated with the sesquilinear form instead of an inner product, and the resulting functions are shown to generate the proper operations. These well-defined regularizations of quasi-derivatives are thus suitable to build a boundary triple.

The use of pieces of solutions from \cite{FL} instead of the known full solutions is greatly beneficial both for forming intuition and for calculations. Ostensibly, the full solutions can generate operations and because they form a basis for the defect spaces they should be able to form a boundary triple. In practice, these operations appear to be more ``degenerate'' than ours and are also very difficult to deal with in calculations, e.g.~simply plugging two solutions into a sesquilinear form associated with a general power seems unfeasible. Subsection \ref{ss-otherexamples} shows that these two methods yield the same operations in the uncomposed case.

The Jacobi differential operator having two limit-circle endpoints means that the spectrum is discrete, and in this case simple, so the spectrum for the powers of the operator can be inferred in some instances. The obtained Weyl $m$-functions for powers of the operator therefore don't yield surprising information in the most common cases. The ease with which they are obtained does allow for interesting examples, such as separated and periodic boundary conditions though. Additionally, the method that produces the quasi-derivatives allows for the Weyl $m$-function to be determined for powers of other Sturm--Liouville operators, which can possess more complicated spectra, and possibly for more general higher-order ordinary differential equations. The only inhibiting factors to such generalizations are proving some structural results of \cite{EKLWY} and \cite{FL} for the operator (or class of operators) of interest. However

%%%%%%%%%%%%%%%%%%%%%%%%%%%%%
%%%%%%%%%%%%%%%%%%%%%%%%%%%%%
\subsection{Outline}
%%%%%%%%%%%%%%%%%%%%%%%%%%%%%
%%%%%%%%%%%%%%%%%%%%%%%%%%%%%

Section \ref{s-bg} introduces tools from different fields that concern boundary conditions for self-adjoint extensions. Sturm--Liouville operators and their powers are briefly discussed and then the classical framework for self-adjoint extensions, attributed collectively to Glazman--Krein--Naimark, is presented in Subsection \ref{ss-extensions}. Prerequisite facts and definitions in the theory of boundary triples are recalled from \cite{BdS} in Subsection \ref{ss-bt}.

Section \ref{s-jacobi} illustrates how boundary triples can be used to compute the Weyl $m$-function by focusing on the classical Jacobi differential operator. Surprisingly, the resulting $2\times 2$ matrix-valued Nevanlinna--Herglotz function seems to be new to the literature. The maps that form the boundary triple are generated by full solutions but are shown to be equivalent to those created by just pieces of solutions in Subsection \ref{ss-otherexamples}. The operations and boundary triples for the Legendre and Laguerre differential operators are also presented.

Section \ref{s-powers} contains the main results of the paper. Pieces of solutions to the Jacobi differential equation are shown to define some quasi-derivatives when placed into the sesquilinear form, but those coming from non-principal solutions require some alterations. They are put through a modified Gram--Schmidt process in Subsection \ref{ss-regularizations} to help make their impacts not overlap with one another, and when put into the sesquilinear form these new functions yield regularizations of quasi-derivatives. Subsection \ref{ss-natural} concludes that these operations create a boundary triple for the associated maximal domain.

The general theory of boundary triples is used on the setup of the previous Section to obtain explicit Weyl $m$-functions for specific self-adjoint extensions in Section \ref{s-mfunctions}, including the Friedrichs extension. General expressions are also available to represent all possible self-adjoint extensions. Subsection \ref{ss-otherbc} shows how the chosen boundary triple can be transformed to represent other interesting boundary conditions, including separated and periodic.

%%%%%%%%%%%%%%%%%%%%%%%%%%%%%
%%%%%%%%%%%%%%%%%%%%%%%%%%%%%
\section{Background}\label{s-bg}
%%%%%%%%%%%%%%%%%%%%%%%%%%%%%
%%%%%%%%%%%%%%%%%%%%%%%%%%%%%

Consider the classical Sturm--Liouville differential equation
\begin{align}\label{d-sturmdif}
\dfrac{d}{dx}\left[p(x)\dfrac{df}{dx}(x)\right]+q(x)f(x)=-\lambda w(x)f(x),
\end{align}
where $p(x),w(x)>0$ a.e.~on $(a,b)$ and $q(x)$ real-valued a.e.~on $(a,b)$, with $a,b\in\RR\cup\{\pm \infty\}$. 
Furthermore, $1/p(x),q(x),w(x)\in L^1\ti{loc}[(a,b),dx]$. Additional details about Sturm--Liouville theory can be found in \cite{AG, BEZ, E, GZ, Z}.
The differential expression can be viewed as a linear operator, mapping a function $f$ to the function $\ell[f]$ via
\begin{align}\label{d-sturmop}
\ell[f](x):=-\dfrac{1}{w(x)}\left(\dfrac{d}{dx}\left[p(x)\dfrac{df}{dx}(x)\right]+q(x)f(x)\right).
\end{align}
This unbounded operator acts on the Hilbert space $L^2[(a,b),w]$, endowed with the inner product 
$
\langle f,g\rangle:=\int_a^b f(x)\overline{g(x)}w(x)dx.
$
In this setting, the eigenvalue problem $\ell[f](x)=\lambda f(x)$ can be considered. However, the operator acting via $\ell[\fdot]$ on $L^2[(a,b),w]$ is not self-adjoint a priori. Additional boundary conditions are required to ensure this property.

Additionally, the operator $\ell^n[\fdot]$ is defined as the operator $\ell[\fdot]$ composed with itself $n$ times, creating a differential operator of order $2n$. Every formally symmetric differential expression $\ell^n[\fdot]$ of order $2n$ with coefficients $a_k:(a,b)\to\RR$ and $a_k\in C^k(a,b)$, for $k=0,1,\dots,n$ and $n\in\NN$, has the {\bf Lagrangian symmetric form} 
\begin{align}\label{e-lagrangian}
\ell^n[f](x)=\sum_{j=1}^n(-1)^j(a_j(x)f^{(j)}(x))^{(j)}, \text{ } x\in(a,b).
\end{align}
Further details can be found in \cite{DS, ELT, LWOG}.

%%%%%%%%%%%%%%%%%%%%%%%%%%%%%
%%%%%%%%%%%%%%%%%%%%%%%%%%%%%
\subsection{Extension Theory}\label{ss-extensions}
%%%%%%%%%%%%%%%%%%%%%%%%%%%%%
%%%%%%%%%%%%%%%%%%%%%%%%%%%%%

There is a vast amount of literature concerning the extensions of symmetric operators. Here we present only that which pertains to self-adjoint extensions.

\begin{defn}[variation of {\cite[Section 14.2]{N}}]\label{d-defect}
For a a symmetric, closed operator ${\bf  A}$ on a Hilbert space $\cH$, define 
the {\bf positive defect space} and the {\bf negative defect space}, respectively, by
$$\cD_+:=\left\{f\in\cD({\bf  A}^*)~:~{\bf  A}^*f=if\right\}
\qquad\text{and}\qquad
\cD_-:=\left\{f\in\cD({\bf  A}^*)~:~{\bf  A}^*f=-if\right\}.$$
\end{defn}

Note that the self-adjoint extensions of a symmetric operator coincide with those of the closure of the symmetric operator {\cite[Theorem XII.4.8]{DS}}, so without loss of generality we assume that all considered operators are closed.

The dimensions dim$(\cD_+)=m_+$ and dim$(\cD_-)=m_-$, called the {\bf positive} and {\bf negative deficiency indices of ${\bf  A}$} respectively, will play an important role. They are usually conveyed as the pair $(m_+,m_-)$. 
The deficiency indices of $T$ correspond to how ``far'' from self-adjoint ${\bf  A}$ is. A symmetric operator ${\bf  A}$ has self-adjoint extensions if and only if its deficiency indices are equal {\cite[Section 14.8.8]{N}}.

\begin{theo}[{\cite[Theorem 14.4.4]{N}}]\label{t-decomp}
If ${\bf  A}$ is a closed, symmetric operator, then the subspaces $\cD({\bf  A})$, $\mathcal{D}_+$, and $\mathcal{D}_{-}$ are linearly independent and their direct sum coincides with $\cD({\bf  A}^*)$, i.e.,
$$\cD({\bf  A}^*)=\cD({\bf  A})\dotplus\mathcal{D}_+ \dotplus\mathcal{D}_{-}.$$
(Here, subspaces $\cX_1, \cX_2, \hdots ,\cX_p$ are said to be {\bf linearly independent}, if $\sum_{i=1}^p x_i = 0$ for $x_i\in \cX_i$ implies that all $x_i=0$.)
\end{theo}

We now let $\ell[\fdot]$ be a Sturm--Liouville differential expression in order to introduce more specific definitions. It is important to reiterate that the analysis of self-adjoint extensions does not involve changing the differential expression associated with the operator at all, merely the domain of definition by applying boundary conditions. 

\begin{defn}[{\cite[Section 17.2]{N}}]\label{d-max}
The {\bf maximal domain} of $\ell[\fdot]$ is given by 
\begin{align*}
\cD\ti{max}=\cD\ti{max}(\ell):=\left\{f:(a,b)\to\mathbb{C}~:~f,pf'\in\text{AC}\ti{loc}(a,b);
f,\ell[f]\in L^2[(a,b),w]\right\}.
\end{align*}
\end{defn}

The designation of ``maximal'' is appropriate in this case because $\cD\ti{max}(\ell)$ is the largest possible subspace that $\ell$ maps back into $L^2[(a,b),w]$. For $f,g\in\cD\ti{max}(\ell)$ and $a<\al\le \beta<b$ the {\bf sesquilinear form} associated with $\ell$ by 
\begin{equation}\label{e-greens}
[f,g]\bigg|_{\al}^{\beta}:=\int_{\al}^{\beta}\left\{\ell[f(x)]\overline{g(x)}-\ell[\overline{g(x)}]f(x)\right\}w(x)dx.
\end{equation}

\begin{theo}[{\cite[Section 17.2]{N}}]\label{t-limits}
The limits $[f,g](b):=\lim_{x\to b^-}[f,g](x)$ and $[f,g](a):=\lim_{x\to a^+}[f,g](x)$ exist and are finite for $f,g\in\cD\ti{max}(\ell)$.
\end{theo}

The equation \eqref{e-greens} is {\bf Green's formula} for $\ell[\fdot]$, and in the case of Sturm--Liouville operators it can be explicitly computed using integration by parts to be the modified Wronskian
\begin{align}\label{e-mwronskian}
[f,g]\bigg|_a^b:=p(x)[f'(x)g(x)-f(x)g'(x)]\bigg|_a^b.
\end{align}

\begin{defn}[{\cite[Section 17.2]{N}}]\label{d-min}
The {\bf minimal domain} of $\ell[\fdot]$ is given by
\begin{align*}
\cD\ti{min}=\cD\ti{min}(\ell):=\left\{f\in\cD\ti{max}(\ell)~:~[f,g]\big|_a^b=0~~\forall g\in\cD\ti{max}(\ell)\right\}.
\end{align*}
\end{defn}

The maximal and minimal operators associated with the expression $\ell[\fdot]$ are then defined as ${\bf L}\ti{min}=\{\ell,\cD\ti{min}\}$ and ${\bf L}\ti{max}=\{\ell,\cD\ti{max}\}$ respectively. By {\cite[Section 17.2]{N}}, these operators are adjoints of one another, i.e.~$({\bf L}\ti{min})^*={\bf L}\ti{max}$ and $({\bf L}\ti{max})^*={\bf L}\ti{min}$.

In the context of differential operators, Theorem \ref{t-decomp} can be restated.

\begin{theo}[{\cite[Section 14.5]{N}}]\label{t-neumanndecomp}
Let $\cD\ti{max}$ and $\cD\ti{min}$ be the maximal and minimal domains associated with the differential expression $\ell[\fdot]$, respectively. Then, 
\begin{equation}\label{e-vN}
\cD\ti{max}=\cD\ti{min}\dotplus\cD_+\dotplus\cD_-.
\end{equation}
\end{theo}

Equation \eqref{e-vN} is commonly known as {\bf von Neumann's formula}. Here $\dotplus$ denotes the direct sum, and $\cD_+,\cD_-$ are the defect spaces associated with the expression $\ell[\fdot]$. The decomposition can be made into an orthogonal direct sum by using the graph norm, see \cite{FFL}.

If the operator ${\bf L}\ti{min}$ acts via an expression $\ell[\fdot]$ of order $n$ and has any self-adjoint extensions, then the deficiency indices of ${\bf L}\ti{min}$ have the form $(m,m)$, where $0\leq m\leq n$ {\cite[Section 14.8.8]{N}}.
Hence, Sturm--Liouville expressions that generate self-adjoint operators have deficiency indices $(0,0)$, $(1,1)$ or $(2,2)$. If a differential expression is either in the limit-circle case or regular at the endpoint $a$, it requires a boundary condition at $a$. If it is in the limit-point case at the endpoint $a$, it does not require a boundary condition. The analogous statements are true at the endpoint $b$. These facts can be summed up in the following result. 

\begin{theo}\label{t-endpointnumbers}
Let ${\bf L}\ti{min}=\{\ell,\cD\ti{min}\}$, where $\ell$ is a singular Sturm--Liouville differential expression.
\[ 
m_{\pm}({\bf L}\ti{min})=
\begin{cases} 
      2 & \text{if }\ell\text{ is limit-circle at }a\text{ and }b, \\
      1 & \text{if }\ell\text{ is limit-circle at }a\text{ and limit-point at }b\text{ or vice versa}, \\
      0 & \text{if }\ell\text{ is limit-point at }a\text{ and }b. 
   \end{cases}
\]
\end{theo}

Sturm--Liouville differential expressions are extremely well-researched, see e.g.~\cite{BEZ,E} for an encyclopedic reference, so the deficiency indices are well-known in almost all cases of interest. Representative examples are the Jacobi operator, with deficiency indices $(2,2)$, the Laguerre operator, with indices $(1,1)$, and the Hermite operator, with indices $(0,0)$. Jacobi operators are in the limit-circle case at both -1 and 1 (for $\al,\beta$ both in $[0,1)$), Laguerre operators are in the limit-circle case at 0 and the limit point case at $\infty$ (for $\al\in[0,1)$ and $\al^2\neq 1/2$), and Hermite operators are in the limit point case at both $\pm\infty$. Hermite operators are thus essentially self-adjoint and require no boundary conditions. Since we are primarily concerned with only the Jacobi operator this information will suffice, but any of \cite{BEZ, GZ, NZ, W, Z} can be consulted for more information on the classification of endpoints.

The following Theorem explicitly shows how the defect spaces impact self-adjoint extensions. To this end, let $\f_j$, for $j=1,\dots m$, denote an orthonormal basis of $\cD_+$. The functions $\overline{\f_j}$ are thus an orthonormal basis of $\cD_-.$

\begin{theo}[{\cite[Theorem 18.1.2]{N}}]\label{t-gknmatrix}
Every self-adjoint extension ${\bf L}=\{\ell,\cD_{{\bf L}}\}$ of the minimal operator ${\bf L}\ti{min}=\{\ell,\cD^n\ti{min}\}$ with deficiency indices $(m,m)$ can be characterized by means of a unitary $m\times m$ matrix $u=[u_{jk}]$ in the following way:

Its domain of definition $\cD_{{\bf L}}$ is the set of all functions $z(x)$ of the form
$$ z(x)=y(x)+\psi(x), $$
where $y(x)\in\cD\ti{min}$ and $\psi(x)$ is a linear combination of the functions
$$ \psi_j(x)=\f_j(x)+\sum_{k=1}^m u_{kj}\overline{\f_k(x)}, \quad j=1,\dots,m. $$
Conversely, every unitary $m\times m$ matrix $u=[u_{jk}]$ determines in the way described above a certain self-adjoint extension ${\bf L}$ of the operator ${\bf L}\ti{min}$. The correspondence thus established between ${\bf L}$ and $u$ is one-to-one.
\end{theo}

In order to formulate the core Glazman--Krein--Naimark (GKN) Theorems, we recall an generalization of linear independence to one that mods out by a subspace. This subspace will be the minimal domain in applications.

\begin{defn}[{\cite[Section 14.6]{N}}]\label{d-linind}
Let $\cX_1$ and $\cX_2$ be subspaces of a vector space $\cX$ such that $\cX_1\le \cX_2$. Let $\{x_1,x_2,\dots,x_r\}\subseteq \cX_2$. We say that $\{x_1,x_2,\dots,x_r\}$ is {\bf linearly independent modulo $\cX_1$} if
$$\sum_{i=1}^r\al_ix_i\in \cX_1 \text{ implies } \al_i=0\text{ for all }i=1,2,\dots, r.$$
\end{defn}

\begin{theo}[GKN1,~{\cite[Theorem 18.1.4]{N}}]\label{t-gkn1}
Let ${\bf L}=\{\ell,\cD_{{\bf L}}\}$ be a self-adjoint extension of the minimal operator ${\bf L}\ti{min}=\{\ell,\cD\ti{min}\}$ with deficiency indices $(m,m)$. Then the domain $\cD_{{\bf L}}$ consists of the set of all functions $f\in\cD\ti{max}$, which satisfy the conditions
\begin{equation}\label{e-gkn1a}
[f,w_k]\bigg|_a^b=0, \text{ }k=1,2,\dots,m ,
\end{equation}
where $w_1,\dots,w_m\in \cD\ti{max}$ are linearly independent modulo $\cD\ti{min}$ for which the relations
\begin{equation}\label{e-gkn1b}
[w_j,w_k]\bigg|_a^b=0, \text{ }j,k=1,2,\dots,m
\end{equation}
hold.
\end{theo}

The requirements in equation \eqref{e-gkn1b} are commonly referred to as {\bf Glazman symmetry conditions}. The converse of the GKN1 Theorem is also true.

\begin{theo}[GKN2,~{\cite[Theorem 18.1.4]{N}}]\label{t-gkn2}
Assume we are given arbitrary functions $w_1,w_2,\dots,w_m\in\cD\ti{max}$ which are linearly independent modulo $\cD\ti{min}$ and which satisfy the relations \eqref{e-gkn1b}. Then the set of all functions $f\in\cD\ti{max}$ which satisfy the conditions \eqref{e-gkn1a} is domain of a self-adjoint extension of ${\bf L}\ti{min}$.
\end{theo}

These two theorems completely answer the question of how boundary conditions can be used to create self-adjoint extensions. Applications of this theory hinge on determining the proper $w_k$'s that will define the domain of a desired self-adjoint extension.

%%%%%%%%%%%%%%%%%%%%%%%%%%%%%
%%%%%%%%%%%%%%%%%%%%%%%%%%%%%
\subsection{Boundary Triples}\label{ss-bt}
%%%%%%%%%%%%%%%%%%%%%%%%%%%%%
%%%%%%%%%%%%%%%%%%%%%%%%%%%%%

The main tool used for calculating the Weyl $m$-function of examples will be boundary triples. Most of the material from this subsection is taken from an excellent book of Jussi Behrndt, Seppo Hassi, and Henk de Snoo \cite{BdS}, which should be consulted for more details. In particular, boundary triples are usually formulated not only for operators but for more general linear relations. 

\begin{defn}{\cite{BdS}}
Let $\fh$ and $\fK$ be Hilbert spaces over $\CC$. A linear subspace of $\fh\times\fK$ is called a {\bf linear relation} $H$ from $\fh$ to $\fK$ and the elements $\widehat{h}\in H$ will in general be written as pairs $\{h,h'\}$ with components $h\in\fh$ and $h'\in\fK$. If $\fh=\fK$ then we will just say $H$ is a linear relation in $\fh$. 
\end{defn}

Linear relations will play a large role in the determination of self-adjoint extensions in Section \ref{s-mfunctions}.

\begin{defn}{\cite[Definition 2.1.1]{BdS}}\label{d-bt}
Let $S$ be a closed symmetric relation in a Hilbert space $\fh$. Then $\{\cG,\Gamma_0,\Gamma_1\}$ is a \textbf{boundary triple} for $S^*$ if $\cG$ is a Hilbert space and $\Gamma_0,\Gamma_1:S^*\to\cG$ are linear mappings such that the mapping $\Gamma:S^*\to\cG\times\cG$ defined by
\begin{align*}
    \Gamma\widehat{f}=\{\Gamma_0\widehat{f},\Gamma_1\widehat{f}\}, ~~~~ \widehat{f}=\{f,f'\}\in S^*,
\end{align*}
is surjective and the identity 
\begin{align}\label{e-btgreens}
    \langle f',g\rangle_{\fh}-\langle f,g'\rangle_{\fh}=\langle\Gamma_1\widehat{f},\Gamma_0\widehat{g}\rangle_{\cG}-\langle\Gamma_0\widehat{f},\Gamma_1\widehat{g}\rangle_{\cG}
\end{align}
holds for all $\widehat{f}=\{f,f'\},\widehat{g}=\{g,g'\}\in S^*$.
\end{defn}

Notice that when $S$ is a Sturm--Liouville differential operator the left-hand side of equation \eqref{e-btgreens} is just the sesquilinear form given in equation \eqref{e-greens}.

The eigenspace of closed symmetric relation $S$ at $\la\in\CC$ will be written as
\begin{align*}
    \mathfrak{N}_{\la}(S^*)=\ker(S^*-\la) ~~\text{ and }~~ \widehat{\mathfrak{N}}_{\la}(S^*)=\left\{\{f_\la,\la f_{\la}\}~:~f_{\la}\in\mathfrak{N}_{\la}(S^*)\right\}.
\end{align*}
Let $\pi_1$ denote the orthogonal projection from $\fh\times\fh$ onto $\fh\times\{0\}$. Then $\pi_1$ maps $\widehat{\mathfrak{N}}_{\la}(S^*)$ bijectively onto $\mathfrak{N}_{\la}(S^*)$.

\begin{defn}{\cite[Definition 2.3.1]{BdS}}
Let $S$ be a closed symmetric relation in a complex Hilbert space $\fh$, let $\{\cG,\Gamma_0,\Gamma_1\}$ be a boundary triple for $S^*$, and let $A_0=\ker \Gamma_0$. Then
\begin{align*}
    \rho(A_0)\ni\la\mapsto \gamma(\la)=\left\{\{\Gamma_0\widehat{f}_{\la},f_{\la}\}~:~\widehat{f}_{\la}\in\widehat{\mathfrak{N}}_{\la}(S^*)\right\},
\end{align*}
or, equivalently,
\begin{align*}
    \rho(A_0)\ni\la\mapsto \gamma(\la)=\pi_1\left(\Gamma_0\upharpoonright\widehat{\mathfrak{N}}_{\la}(S^*)\right)^{-1},
\end{align*}
is called the \textbf{$\gamma$-field} associated with the boundary triple $\{\cG,\Gamma_0,\Gamma_1\}$.
\end{defn}

The structure of boundary triples allows for the classical Weyl $m$-function to be obtained via a simple formula.

\begin{defn}{\cite[Definition 2.3.4]{BdS}}\label{d-mfunction}
Let $S$ be a closed symmetric relation in a complex Hilbert space $\fh$, let $\{\cG,\Gamma_0,\Gamma_1\}$ be a boundary triple for $S^*$, and let $A_0=\ker \Gamma_0$. Then
\begin{align*}
    \rho(A_0)\ni\la\mapsto M(\la)=\left\{\{\Gamma_0\widehat{f}_{\la},\Gamma_1\widehat{f}_{\la}\}~:~\widehat{f}_{\la}\in\widehat{\mathfrak{N}}_{\la}(S^*)\right\},
\end{align*}
or, equivalently,
\begin{align*}
    \rho(A_0)\ni\la\mapsto M(\la)=\Gamma_1\left(\Gamma_0\upharpoonright\widehat{\mathfrak{N}}_{\la}(S^*)\right)^{-1},
\end{align*}
is called the \textbf{Weyl $m$-function} associated with the boundary triple $\{\cG,\Gamma_0,\Gamma_1\}$.
\end{defn}

A closed symmetric relation $S$ in $\fh$ with equal defect indices (analogous to Definition \ref{d-defect}) will admit self-adjoint extensions in $\fh$, each of which will give rise to a boundary triple for $S^*$ via \cite[Theorem 2.4.1]{BdS}. Hence, boundary triples for $S^*$ are not usually unique. Transforming boundary triples therefore plays an important role in obtaining desired sets of boundary conditions. Define the unitary and self-adjoint operator 
\begin{align*}
    \cJ_{\fh}:=\left(\begin{array}{cc}
    0 & -i\cI_{\fh} \\
    -i\cI_{\fh} & 0 
    \end{array}\right),
\end{align*}
on the product space $\fh\times\fh$, where $\cI_{\fh}$ denotes the identity operator in $\cH$. 

\begin{theo}{\cite[Theorem 2.5.1]{BdS}}\label{t-weyltransform}
Let $S$ be a closed symmetric relation in $\fh$, assume that $\{\cG,\Gamma_0,\Gamma_1\}$ is a boundary triple for $S^*$, and let $\cG'$ be a Hilbert space. Then the following statements hold:
\begin{enumerate}
    \item[(i)] Let $\cW$ be a bounded linear operator from $\cG\times\cG$ to $\cG'\times\cG'$ such that
    \begin{align}\label{e-wrelations}
        \cW^*\cJ\ci{\cG'}\cW=\cJ\ci{\cG} \hspace{.5cm}\text{ and }\hspace{.5cm} \cW\cJ\ci{\cG}\cW^*=\cJ\ci{\cG'},
    \end{align}
    and define
    \begin{align}\label{e-wbtrelation}
        \left( \begin{array}{c}
\Gamma_0' \\
\Gamma_1'
\end{array} \right)
=\cW
    \left( \begin{array}{c}
\Gamma_0 \\
\Gamma_1
\end{array} \right)
= \left( \begin{array}{cc}
W_{11} & W_{12} \\
W_{21} & W_{22} 
\end{array} \right)
\left( \begin{array}{c}
\Gamma_0 \\
\Gamma_1
\end{array} \right).
    \end{align}
    Then $\{\cG',\Gamma_0',\Gamma_1'\}$ is a boundary triple for $S^*$.
    \item[(ii)] Let $\{\cG',\Gamma_0',\Gamma_1'\}$ be a boundary triple for $S^*$. Then there exists a unique bounded linear operator $\cW$ from $\cG\times\cG$ to $\cG'\times\cG'$ satisfying equation \eqref{e-wrelations} such that equation \eqref{e-wbtrelation} holds.
\end{enumerate}
\end{theo}

Let $\{\CC^{2n},\Gamma_0,\Gamma_1\}$ be a boundary triple for an operator ${\bf A}$ and the deficiency indices of the associated minimal domain $\cD\ti{min}$ be $(2n,2n)$. Then self-adjoint extensions $A_{\te}\subset\cD\ti{max}$ are in one-to-one correspondence with the self-adjoint relations $\theta\in\CC^{2n}$ via
\begin{align*}
    \dom {\bf A}_{\te}=\left\{f\in\cD\ti{max}~:~\{\Gamma_0,\Gamma_1\}\in\te\right\}.
\end{align*}
Hence, assume that $\te$ is a self-adjoint relation in $\CC^{2n}$. According to \cite[Corollary 1.10.9]{BdS}, the relation $\te$ can be represented with $2n\times 2n$ matrices $\cA$ and $\cB$ satisfying the conditions $\cA^*\cB=\cB^*\cA$, $\cA\cB^*=\cB\cA^*$ and $\cA\cA^*+\cB\cB^*=I=\cA^*\cA+\cB^*\cB$ such that 
\begin{align*}
    \te=\left\{\{\cA\f,\cB\f\} ~:~ \f\in\CC^{2n}\right\}=\left\{\{\psi,\psi'\} ~:~ \cA^*\psi'=\cB^*\psi\right\}.
\end{align*}
In that case, one has
\begin{align}\label{e-thetadomain1}
\dom {\bf A}_{\te}=\left\{ f\in\cD\ti{max} ~:~ \cA^*\Gamma_1(f)=\cB^*\Gamma_0(f)\right\}.
\end{align}
Theorem 2.6.1 and Corollary 2.6.3 from \cite{BdS} then say that for $\la\in\rho({\bf A}_{\te})\cap\rho({\bf A}_0)$ the Krein formula for the corresponding resolvents are given by
\begin{equation}\label{e-thetaresolvent}
\begin{aligned}
(A_{\te}-\la)^{-1}&=(A_0-\la)^{-1}+\gamma(\la)(\te-M(\la))^{-1}\gamma(\overline{\la})^* \\
&=(A_0-\la)^{-1}+\gamma(\la)\cA(\cB-M(\la)\cA)^{-1}\gamma(\overline{\la})^*.
\end{aligned}
\end{equation}
In the case of the examples in this manuscript, the spectrum of ${\bf A}_0$ is discrete and the difference of the resolvents of ${\bf A}_0$ and ${\bf A}_{\te}$ is an operator of rank $\leq$ $2n$. Thus, the spectrum of the self-adjoint operator ${\bf A}_{\te}$ is also discrete. Indeed, $\la\in\rho({\bf A}_0)$ is an eigenvalue of ${\bf A}_{\te}$ if and only if $\ker (\te-M(\la))$, or equivalently, $\ker(\cB-M(\la)\cA)$ is nontrivial, and that 
\begin{align*}
    \ker({\bf A}_{\te}-\la)=\gamma(\la)\ker(\te-M(\la))=\gamma(\la)\cA\ker(\cB-M(\la)\cA).
\end{align*}
In the special case that the self-adjoint relation $\te$ in $\CC^{2n}$ is a $2n\times 2n$ matrix, the boundary condition for the domain of ${\bf A}_{\te}$ can be written as
\begin{align}\label{e-thetadomain2}
    \dom {\bf A}_{\te}=\left\{ f\in\cD\ti{max} ~:~ \te\Gamma_0(f)=\Gamma_1(f)\right\}.
\end{align}
The spectral properties of the operator ${\bf A}_{\te}$ can also be described with the help of the function 
\begin{align}\label{e-thetam}
    \la\mapsto(\te-M(\la))^{-1};
\end{align}
the poles of the matrix function \eqref{e-thetam} coincide with the discrete spectrum of ${\bf A}_{\te}$ and the dimension of the eigenspace $\ker({\bf A}_{\te}-\la)$ coincides with the dimension of the range of the residue of the function \eqref{e-thetam} at $\la$. 

We now let ${\bf L}$ be a Sturm--Liouville operator defined on a subset $(a,b)$ of $\RR\cup\{\pm\infty\}$ with associated maximal domain $\cD\ti{max}$ and sesquilinear form $[\cdot,\cdot]_{\bf L}$. The boundary triples for ${\bf L}$ will be formed with quasi-derivatives.

\begin{defn}\label{d-quasideriv}
Let $u$ and $v$ be linearly independent real solutions of the equation $({\bf L}-\lambda_0)y=0$ for some $\la_0\in\RR$ and assume that the solutions are normalized by $[u,v]_{{\bf L}}=1$. Let $f$ be a complex function on $(a,b)$ for which $f,pf'\in AC(a,b)$. Then the {\bf quasi-derivatives} of $f$, induced by the normalized solutions $u$ and $v$, are defined as complex functions on $(a,b)$ given by
\begin{align*}
f^{[0]}:=[f,v]_{{\bf L}} \text{ and } f^{[1]}:=-[f,u]_{{\bf L}}.
\end{align*}
\end{defn}

Fix a fundamental system $(u_1(\cdot,\la);~u_2(\cdot,\la))$ for the equation $({\bf L}-\la)f=0$ by the initial conditions
\begin{align}
    \left( \begin{array}{cc}
u_1^{[0]}(a,\la) &  u_2^{[0]}(a,\la) \\
u_1^{[1]}(a,\la) & u_2^{[1]}(a,\la)
\end{array} \right)
=
    \left( \begin{array}{cc}
1 & 0 \\
0 & 1
\end{array} \right).
\end{align}
Recall that for all $f\in\cD\ti{max}$ the quasi-derivatives $f^{[0]}(a)$, $f^{[1]}(a)$, $f^{[0]}(b)$ and $f^{[1]}(b)$ are well-defined due to Theorem \ref{t-limits}.

\begin{prop}{\cite[Proposition 6.3.8]{BdS}}\label{t-btlcgeneral}
Assume that the endpoints $a$ and $b$ are in the limit-circle case. Then $\{\CC^2,\Gamma_0,\Gamma_1\}$, where
\begin{align}
    \Gamma_0f:=\left(\begin{array}{c}
f^{[0]}(a)  \\
f^{[0]}(b)
\end{array} \right), \hspace{.5cm}
    \Gamma_1f:=\left( \begin{array}{c}
f^{[1]}(a)   \\
-f^{[1]}(b) 
\end{array} \right), \hspace{.5cm}
f\in\dom T\ti{max},
\end{align}
is a boundary triple for $\cD\ti{max}$. The self-adjoint extension ${\bf L}_0$ corresponding to $\Gamma_0$ is the restriction of $\cD\ti{max}$ defined on 
\begin{align*}
    \dom {\bf L}_0=\{f\in\dom \cD\ti{max} ~:~ f^{[0]}(a)=f^{[0]}(b)=0 \},
\end{align*}
and $u_2^{[0]}(b,\lambda)\neq 0$ for all $\lambda\in\rho({\bf L}_0)$. Moreover, the corresponding $\gamma$-field and Weyl function are given by 
\begin{align*}
    \gamma(\lambda)=(\begin{array}{cc}
u_1(\cdot,\lambda) & u_2(\cdot,\lambda)
\end{array})
\dfrac{1}{u_2^{[0]}(b,\lambda)}
\left( \begin{array}{cc}
u_2^{[0]}(b,\lambda) & 0 \\
-u_1^{[0]}(b,\lambda) & 1
\end{array} \right),
\end{align*}
and
\begin{align*}
    M(\lambda)=\dfrac{1}{u_2^{[0]}(b,\lambda)}
    \left( \begin{array}{cc}
-u_1^{[0]}(b,\lambda) & 1 \\
1 & -u_2^{[1]}(b,\lambda)
\end{array} \right),
\end{align*}
for $\lambda\in\rho({\bf L}_0)$.
\end{prop}

The Proposition will be used as a shortcut in the analysis of the uncomposed Jacobi operator in Section \ref{s-jacobi}. The solutions $u$ and $v$ from Definition \ref{d-quasideriv} fall into two disjoint categories.

\begin{defn}{\cite[Definition 6.10.3]{BdS}}\label{d-principal}
Let $({\bf L}-\lambda_0)f=0$ with $\lambda_0\in\RR$ be non-oscillatory at the endpoint $a$ and let $u$ and $v$ be real solutions of $({\bf L}-\lambda_0)f=0$. Then $u$ is said to be {\bf principal} at $a$ if $1/pu^2$ is not integrable at $a$ and $v$ is said to be {\bf non-principal} at $a$ if $1/pv^2$ is integrable at $a$.
\end{defn}

However, in Subsection \ref{ss-otherexamples} and Section \ref{s-powers} we will not rely on solutions to generate the quasi-derivatives from Definition \ref{d-quasideriv}.

%%%%%%%%%%%%%%%%%%%%%%%%%%%%%
%%%%%%%%%%%%%%%%%%%%%%%%%%%%%
\section{The Jacobi Differential Operator}\label{s-jacobi}
%%%%%%%%%%%%%%%%%%%%%%%%%%%%%
%%%%%%%%%%%%%%%%%%%%%%%%%%%%%

Let $0<\al,\beta<1$, and consider the classical Jacobi differential expression given by
\begin{align}\label{e-jacobi}
\ell_{\al,\beta}[f](x)=-\dfrac{1}{(1-x)^{\al}(1+x)^{\beta}}[(1-x)^{\al+1}(1+x)^{\beta+1}f'(x)]'
\end{align}
on the maximal domain
\begin{align*}
\cD^{(\al,\beta)}\ti{max}=\{ f\in L^2_{\al,\beta}(-1,1) ~|~ f,f'\in AC\ti{loc};\ell_{\al,\beta}[f]\in L^2_{\al,\beta}(-1,1)\},
\end{align*}
where the Hilbert space $L^2_{\al,\beta}(-1,1):=L^2\left[(-1,1),(1-x)^{\al}(1+x)^{\beta}\right]$. This maximal domain defines the associated minimal domain given in Definition \ref{d-min}, and the defect indices are $(2,2)$. The specified values of $\al,\beta$ will ensure that the differential expression is in the limit-circle non-oscillating case at both endpoints, and so are assumed throughout. If either parameter is equal to or larger than 1, then all of our conclusions still hold, but some boundary conditions will be satisfied trivially. If either are less than 0, the corresponding endpoint is regular and although it still requires a boundary condition, these are much simpler and don't need the machinery used here. The case where $\al=\beta=0$ describes the Legendre differential equation, and will be discussed briefly in Subsection \ref{ss-otherexamples}.

The Jacobi polynomials $P_m^{(\al,\beta)}(x)$, $m\in\NN_0$, form a complete orthogonal set in $L^2_{\al,\beta}(-1,1)$ for which $f(x)=P_m^{(\al,\beta)}(x)$ solves the eigenvalue equation of the symmetric expression given in equation \eqref{e-jacobi}, that is:
\begin{align}\label{e-jacobieigen}
\ell_{\al,\beta}[f](x)=m(m+\al+\beta+1)f(x)
\end{align}
for each $m$. The Jacobi polynomials can be represented via a Rodrigues' formula as
\begin{align*}
    P_m^{(\al,\beta)}(x)=\dfrac{(-1)^n}{2^nn!}(1-x)^{-\al}(1+x)^{-\beta}\dfrac{\dd^n}{\dd x^n}\left\{(1-x)^{\al}(1+x)^{\beta}(1-x^2)^n\right\},
\end{align*}
and $P_m^{(\al,\beta)}\in\cD\ti{max}^{(\al,\beta)}$.

The associated sesquilinear form is defined, for $f,g\in\cD\ti{max}^{(\al,\beta)}$, via equation \eqref{e-greens}. Integration by parts easily yields the explicit expression
\begin{align*}
[f,g]_1(\pm 1):=\lim_{x\to\pm 1^{\mp}}(1-x)^{\al+1}(1+x)^{\beta+1}[g'(x)f(x)-f'(x)g(x)].
\end{align*}
Note that the dependence of the sesquilinear form on the parameters $\al$ and $\beta$ is suppressed in the definition for the sake of notation. Theorem \ref{t-limits} also says that the sesquilinear form is both well-defined and finite for all $f,g\in\cD\ti{max}^{(\al,\beta)}$. 

Both solutions to equation \eqref{e-jacobi} are easily given in terms of hypergeometric functions. In order to take advantage of this fact, we begin by transforming the expression into the hypergeometric differential equation
\begin{align}\label{e-hyper1}
    z(1-z)f''(z)+[c-(a+b+1)z]f'(z)-abf(z)=0 \text{ for }z\in\CC,
\end{align}
where, in general $a,b,c\in\CC$. The equation has three regular singular points: $z=0,1,\infty$ \cite[Section 15.10]{DLMF}. This description is powerful because any second-order differential equation with three regular singular points can be converted to the hypergeometric equation by a change of variables. Indeed, the hypergeometric equation is a special case of Riemann's differential equation, which also has three regular singular points.

Hence, rewrite equation \eqref{e-jacobi} and make the change of variables $x=1-2t$ and $t=(1-x)/2$ so that it is equal to
\begin{align}\label{e-fulljacobi}
    \ell_{\al,\beta}[f](t)=t(1-t)f''(t)+[(\al+1)-(\al+\beta+2)t]f'(t)=-\la f(t),
\end{align}
with the spectral parameter $\la\in\CC$. For simplicity, write $\la=\mu(\mu+\al+\beta+1)$ to help denote the natural eigenfunctions. This leads to the connecting formulas between equations \eqref{e-hyper1} and \eqref{e-fulljacobi}:
\begin{align}\label{e-hyperpara}
c=\al+1, \hspace{3em} a+b=\al+\beta+1, \hspace{3em} ab=-\lambda.
\end{align}
There are then two choices for $a$ and $b$, but the slots are interchangeable so we choose $a=-\mu$ and $b=\mu+\al+\beta+1$. See \cite[Section 9]{E} and \cite[Chapter IV, Section 4.18]{Titch} for more. Lastly, we assume throughout our discussion that $a-b$ is not equal to an integer.

Define the Gauss hypergeometric series (or function) as
\begin{align}\label{d-hyperseries}
F(a,b;c;z):=\,_2F_1(a,b;c;z)=F(b,a;c;z)=\sum_{n=0}^{\infty}\dfrac{(a)^{(n)}(b)^{(n)}}{(c)^{(n)}}\dfrac{z^n}{n!},
\end{align}
and $(x)^n=x(x+1)(x+2)\cdots(x+n-1)$ denotes the Pochhammer function, or rising factorial, on the disk $|z|<1$. This notation is usually written as a subscript in the hypergeometric community, but we use the stated form now due to additional uses in other contexts. There is conditional convergence for $|z|=1$, except z=1, due to the fact that $c-a-b=-\beta\in(-1,0]$. We also point out that 
\begin{align}\label{e-hyperat0}
    F(a,b;c;0)=1 \text{ for all }a,b,c\in\CC.
\end{align}
Finally, recall two formulas for derivatives of hypergeometric functions that will be useful:
\begin{equation}\label{e-hyperderiv}
\begin{aligned}
    \frac{d^n}{dz^n}F(a,b;c;z)&=\frac{(a)^{(n)}(b)^{(n)}}{(c)^{(n)}}F(a+n,b+n;c+n;z), \\
    \frac{d^n}{dz^n}[z^{c-1}F(a,b;c;z)]&=(c-n)^{(n)} z^{c-n-1}F(a,b;c-n;z). 
\end{aligned}
\end{equation}
The hypergeometric function has been extensively studied due to a wide range of applications, see \cite[Section 15.2]{DLMF} for basic properties and further references. 

The two linearly independent solutions of the differential equation \eqref{e-fulljacobi} in neighborhoods of $t=0$ and $t=1$ (respectively $x=1$ and $x=-1$) are known \cite[Equations 15.10.2-4]{DLMF} to be
\begin{equation}\label{e-solutions1}
\begin{aligned}
w_1(t,\la)&:=
\left.\begin{cases}
F(-\mu,\mu+\al+\beta+1;\al+1;t) & \text{near }t=0 \\
-F(-\mu,\mu+\al+\beta+1;\beta+1;1-t) & \text{near }t=1
\end{cases}\right\}, \\
w_2(t,\la)&:=\left.\begin{cases}
\dfrac{t^{-\al}}{\al 2^{\al+\beta+1}}F(-\mu-\al,\mu+\beta+1;1-\al;t) & \text{near }t=0 \\
\dfrac{(1-t)^{-\beta}}{\beta 2^{\al+\beta+1}}F(-\mu-\beta,\mu+\al+1;1-\beta;1-t) & \text{near }t=1
\end{cases}\right\},
\end{aligned}
\end{equation}
where the constants will help with normalization later and the dependence on the choice of $\la$ stems from equation \eqref{e-hyperpara}.

However, in order to generate operations it is convenient to take a fixed $\la=\la_0=0$. This choice immediately yields that $a=0=\mu$ and $b=\al+\beta+1$ in equation \eqref{e-hyperpara}. Thus, define
\begin{equation}\label{e-particsolution}
\begin{aligned}
\widetilde{w}_1(t)&:=
\left.\begin{cases}
F(0,\al+\beta+1;\al+1;t) & \text{near }t=0 \\
-F(0,\al+\beta+1;\beta+1;1-t) & \text{near }t=1
\end{cases}\right\}=\left.\begin{cases}
1 & \text{near }t=0 \\
-1 & \text{near }t=1
\end{cases}\right\}, \\
\widetilde{w}_2(t)&:=\left.\begin{cases}
\dfrac{t^{-\al}}{\al 2^{\al+\beta+1}}F(-\al,\beta+1;1-\al;t) & \text{near }t=0 \\
\dfrac{(1-t)^{-\beta}}{\beta 2^{\al+\beta+1}}F(-\beta,\al+1;1-\beta;1-t) & \text{near }t=1
\end{cases}\right\},
\end{aligned}
\end{equation}

The operations, after changing the variable back to $x$, are then generated via these particular solutions:
\begin{align}\label{e-defineoperations}
     f^{[0]}(x):=[f,\widetilde{w}_2]_1(x)  \text{ and } f^{[1]}(x):=[f,\widetilde{w}_1]_1(x).
\end{align}
Indeed, the operations differ by a minus sign from those in Definition \ref{d-quasideriv} but the properties in equations \eqref{e-wronskians1} and \eqref{e-hyperinitial} are sufficient to utilize Proposition \ref{t-btlcgeneral} still.

Equations \eqref{e-hyperat0} and \eqref{e-hyperderiv} can be used to simplify
\begin{equation}\label{e-operations1}
\begin{aligned}
    f^{[0]}(-1)&=[f,\widetilde{w}_2]_1(-1)=\lim_{x\to -1^+}-f(x)-\frac{(1+x)f'(x)}{\beta}, \\
    f^{[0]}(1)&=[f,\widetilde{w}_2]_1(1)=\lim_{x\to 1^-}f(x)-\frac{(1-x)f'(x)}{\al}, \\
    f^{[1]}(-1)&=[f,\widetilde{w}_1]_1(1)=[f,-1]_1(-1)=\lim_{x\to -1^+}-(1-x)^{\al+1}(1+x)^{\beta+1}f'(x), \\
    f^{[1]}(1)&=[f,\widetilde{w}_{1(1)}]_1=[f,1]_1(1)=\lim_{x\to 1^-}(1-x)^{\al+1}(1+x)^{\beta+1}f'(x).
\end{aligned}
\end{equation}
All of the previous limits are guaranteed to exist and be finite by Theorem \ref{t-limits}. Notice that the chosen constants \cite[Equations 15.10.3-5]{DLMF} have the effect of making
\begin{align}\label{e-wronskians1}
    [\widetilde{w}_1,\widetilde{w}_2]_1(-1)=[\widetilde{w}_1,\widetilde{w}_2]_1(1)=[w_1,w_2]_1(-1)=[w_1,w_2]_1(1)=1.
\end{align}

This process easily verifies that the fundamental system of solutions $(w_1(x,\la),w_2(x,\la))$ for the equation $(\ell-\la)f=0$ given in \eqref{e-solutions1} satisfies the initial conditions
\begin{align}\label{e-hyperinitial}
    \left( \begin{array}{cc}
w_1^{[0]}(-1,\la) &  w_2^{[0]}(-1,\la) \\
w_1^{[1]}(-1,\la) & w_2^{[1]}(-1,\la)
\end{array} \right)
=
    \left( \begin{array}{cc}
1 & 0 \\
0 & 1
\end{array} \right).
\end{align}
Furthermore, use the above operations to define
\begin{align}
    \Gamma_0f:=\left(\begin{array}{c}
f^{[0]}(-1)  \\
f^{[0]}(1)
\end{array} \right), \hspace{.5cm}
    \Gamma_1f:=\left( \begin{array}{c}
f^{[1]}(-1)   \\
-f^{[1]}(1) 
\end{array} \right), \hspace{.5cm}
f\in\dom T\ti{max}.
\end{align}

Proposition \ref{t-btlcgeneral} can be invoked to conclude $\{\CC^2,\Gamma_0,\Gamma_1\}$ is a boundary triple for $\cD^{(\al,\beta)}\ti{max}$. Additionally, the self-adjoint extension ${\bf A}_0$ corresponding to $\Gamma_0$ is the restriction of $\cD^{(\al,\beta)}\ti{max}$ to 
\begin{align*}
    \dom {\bf A}_0=\{f\in\dom \cD^{(\al,\beta)}\ti{max} ~:~ f^{[0]}(-1)=f^{[0]}(1)=0 \},
\end{align*}
and $w_2^{[0]}(1,\lambda)\neq 0$ for all $\lambda\in\rho({\bf A}_0)$. The corresponding $\gamma$-field and Weyl function are given by 
\begin{align*}
    \gamma(\lambda)=(\begin{array}{cc}
w_1(\cdot,\lambda) & w_2(\cdot,\lambda)
\end{array})
\dfrac{1}{w_2^{[0]}(1,\lambda)}
\left( \begin{array}{cc}
w_2^{[0]}(1,\lambda) & 0 \\
-w_1^{[0]}(1,\lambda) & 1
\end{array} \right),
\end{align*}
and
\begin{align*}
    M(\lambda)=\dfrac{1}{w_2^{[0]}(1,\lambda)}
    \left( \begin{array}{cc}
-w_1^{[0]}(1,\lambda) & 1 \\
1 & -w_2^{[1]}(1,\lambda)
\end{array} \right),
\end{align*}
for $\lambda\in\rho({\bf A}_0)$. Concrete expressions for the $\gamma$-field and Weyl function can then be found by taking advantage of the initial conditions. It is thus necessary to state connection formulas between the endpoints for the hypergeometric function \cite[Equation 15.8.4]{DLMF}:
\begin{equation}\label{e-connection}
\begin{aligned}
    F(a,b;c;1-z)=&\frac{\pi}{\sin(\pi(c-a-b))\Gamma(c-a)\Gamma(c-b)}F(a,b;a+b-c+1;z) \\
    &-\frac{\pi z^{c-a-b}}{\sin(\pi(c-a-b))\Gamma(a)\Gamma(b)}F(c-a,c-b;c-a-b+1;z),
\end{aligned}
\end{equation}
which holds for $|ph(z)|<\pi$ and $|ph(1-z)|<\pi$, where $ph(z)$ denotes the principal value of $z$. Hence, the solution $w_1(x,\la)$ near the endpoint $x=1$ can be rewritten as
\begin{equation*}
\begin{aligned}
w_1(x,\la)&=F\left(-\mu,\mu+\al+\beta+1;\al+1;\frac{1-x}{2}\right) \\
&=\frac{\pi}{\sin(-\beta\pi)\Gamma(\mu+\al+1)\Gamma(-\mu-\beta)}F\left(-\mu,\mu+\al+\beta+1;\beta+1;\frac{1+x}{2}\right) \\
&-\frac{\pi 2^{\beta} (1+x)^{-\beta}}{\sin(-\beta\pi)\Gamma(-\mu)\Gamma(\mu+\al+\beta+1)}F\left(-\mu-\beta,\mu+\al+1;1-\beta;\frac{1+x}{2}\right).
\end{aligned}
\end{equation*}
Notice that now $w_1(1,\la)=c_1 w_1(-1,\la)+c_2 w_2(-1,\la)$, where $c_1,c_2$ do not depend on $x$. Explicitly,
\begin{equation}\label{e-firstconstants}
\begin{aligned}
c_1=&\frac{-\pi}{\sin(-\beta\pi)\Gamma(\mu+\al+1)\Gamma(-\mu-\beta)}, \\
c_2=&\frac{-\pi\beta 2^{\al+\beta+1}}{\sin(-\beta\pi)\Gamma(-\mu)\Gamma(\mu+\al+\beta+1)}.
\end{aligned}
\end{equation}

The solution $w_2(x,\la)$ near the endpoint $x=1$ can also be rewritten via equation \eqref{e-connection} with the additional help of \cite[Equation 15.8.1]{DLMF} as
\begin{equation*}
\begin{aligned}
w_2(x,\la)&=\dfrac{2^{\al}(1-x)^{-\al}}{\al 2^{\al+\beta+1}}F\left(-\mu-\al,\mu+\beta+1;1-\al;\frac{1-x}{2}\right)\\
&=\frac{\pi}{\al 2^{\al+\beta+1}\sin(-\beta\pi)\Gamma(1+\mu)\Gamma(-\mu-\al-\beta)}F\left(-\mu,\mu+\al+\beta+1;\beta+1;\frac{1+x}{2}\right) \\
&-\frac{\pi 2^{\beta} (1+x)^{-\beta}}{\al 2^{\al+\beta+1}\sin(-\beta\pi)\Gamma(-\mu-\al)\Gamma(\mu+\beta+1)}F\left(-\mu-\beta,\mu+\al+1;1-\beta;\frac{1+x}{2}\right).
\end{aligned}
\end{equation*}
It is now possible to write $w_1(1,\la)=c_3 w_1(-1,\la)+c_4 w_2(-1,\la)$, where $c_3,c_4$ do not depend on $x$. In particular,
\begin{equation}\label{e-secondconstants}
\begin{aligned}
c_3=&\frac{-\pi}{\al 2^{\al+\beta+1}\sin(\al\pi)\Gamma(1+\mu)\Gamma(-\mu-\al-\beta)} \\
c_4=&\frac{-\beta\pi}{\al\sin(\al\pi)\Gamma(-\mu-\al)\Gamma(\mu+\beta+1)}.
\end{aligned}
\end{equation}
These calculations combined with the initial conditions in equation \eqref{e-hyperinitial} allow operations to be easily performed on the solutions. For instance, we have
\begin{equation}
\begin{aligned}
\begin{array}{c c}
w_1^{[0]}(1,\la)=c_1, & w_2^{[0]}(1,\la)=c_3, \\
w_1^{[1]}(1,\la)=c_2, & w_2^{[1]}(1,\la)=c_4. 
\end{array}
\end{aligned}
\end{equation}
Therefore, the self-adjoint extension ${\bf A}_0$ with domain
\begin{align*}
    \dom {\bf A}_0=\{f\in\dom \cD^{(\al,\beta)}\ti{max} ~:~ f^{[0]}(-1)=f^{[0]}(1)=0 \},
\end{align*}
thus has $\gamma$-field and Weyl function, respectively,
\begin{align*}
    \gamma(\lambda)=(\begin{array}{cc}
w_1(\cdot,\lambda) & w_2(\cdot,\lambda)
\end{array})
\dfrac{1}{c_3}
\left( \begin{array}{cc}
c_3 & 0 \\
-c_1 & 1
\end{array} \right)
\end{align*}
and
\begin{align*}
    M(\lambda)=\dfrac{1}{c_3}
    \left( \begin{array}{cc}
-c_1 & 1 \\
1 & -c_4
\end{array} \right),
\end{align*} 
for $\lambda\in\rho({\bf A}_0)$. Notice that the function $1/\Gamma(z)$ is analytic for all $z\in\CC$. However, $\Gamma(z)$ has simple poles when $z$ is a non-positive integer, and is analytic otherwise. The spectrum of ${\bf A}_0$ can thus be extracted from the $1/c_3$ term in the Weyl function. In particular, if $1+\mu=-m$, for $m\in\NN_0$, then $\mu=-m-1$. Alternatively, if $-\mu-\al-\beta=-m$, then $\mu=m-\al-\beta$ and $a=\al+\beta-m$. In either case, we have 
\begin{align*}
    \la=\mu(\mu+\al+\beta+1)=(-m-1)(-m+\al+\beta)=(m+1)(m-\al-\beta),
\end{align*}
for $m\in\NN_0$. The lowest eigenvalue of the self-adjoint extension is therefore $-\al-\beta$, keeping in mind that $\al,\beta\in(0,1)$. 

The domain of ${\bf A}_0$ is just the kernel of the operation $\Gamma_0$, and the imposed conditions represent the limit-circle analog of Dirichlet boundary conditions. We can interchange our definitions of $\Gamma_0$ and $\Gamma_1$ to yield another important self-adjoint extension, call it ${\bf A}_1$, with domain that is the kernel of $\Gamma_1$. This extension has boundary conditions that are the analog of von Neumann conditions. The boundary triple is still clearly well-defined, and solutions and operations in the $\gamma$-field and Weyl function are all switched. Respectively, these are now given by
\begin{align*}
    \gamma(\lambda)=(\begin{array}{cc}
w_1(\cdot,\lambda) & w_2(\cdot,\lambda)
\end{array})
\dfrac{1}{c_2}
\left( \begin{array}{cc}
c_2 & 0 \\
-c_4 & 1
\end{array} \right)
\end{align*}
and
\begin{align*}
    M(\lambda)=\dfrac{1}{c_2}
    \left( \begin{array}{cc}
-c_4 & 1 \\
1 & -c_1
\end{array} \right)
\end{align*} 
for $\lambda\in\rho({\bf A}_1)$. Spectral information about ${\bf A}_1$ can be extracted from the $1/c_2$ term and there are poles when $\mu=m$ and $\mu=m+\al+\beta+1$ for $m\in\NN_0$. Hence, simple eigenvalues of ${\bf A}_1$ occur at
\begin{align*}
    \la=m(m+\al+\beta+1) \text{ for } m\in\NN_0.
\end{align*}
The lowest eigenvalue of ${\bf A}_1$ is $0$. Corollary 6.11.9(iii) of \cite{BdS} identifies ${\bf A}_1$ as the important Friedrichs extension of the minimal operator, and gives a few other descriptions. This set of eigenvalues is completely generated by the Jacobi polynomials, see equation \eqref{e-jacobieigen}, and hence ${\bf A}_1$ is the classical self-adjoint extension that is commonly used in the literature. 

A discussion of more general self-adjoint extensions for the minimal domain $\cD\ti{min}^{(\al,\beta)}$ is avoided here for brevity. As a consolation, this matter will be discussed in the more difficult case of powers of the Jacobi differential operator in Subsection \ref{ss-otherbc}.

%%%%%%%%%%%%%%%%%%%%%%%%%%%%%
%%%%%%%%%%%%%%%%%%%%%%%%%%%%%
\subsection{Remarks on Generating Functions and Other Examples}\label{ss-otherexamples}
%%%%%%%%%%%%%%%%%%%%%%%%%%%%%
%%%%%%%%%%%%%%%%%%%%%%%%%%%%% 

The operations defined in equation \eqref{e-operations1} were essential to building the boundary triple for the Jacobi differential operation and beg for a more nuanced analysis. Some other examples of differential equations are also briefly mentioned. 

Begin by defining a $C^2$ function, which is clearly in $\cD\ti{max}^{(\al,\beta)}$:
\begin{align}\label{e-vdefn}
    v(x):=\left.\begin{cases}
    \dfrac{(1+x)^{-\beta}}{\beta 2^{\al+1}} & \text{ for }x\text{ near }-1 \\
    \dfrac{(1-x)^{-\al}}{\al 2^{\beta+1}} & \text{ for }x\text{ near }1
    \end{cases}\right\},
\end{align}
so that another operation can be defined as $f^{*}(x):=[f,v](x)$. Observe that, upon inspection, 
\begin{align*}
    f^{*}(-1)=f^{[0]}(-1), ~~\text{ and }~~ f^{*}(1)=f^{[0]}(1),
\end{align*}
with the operation $f^{[0]}$ given by equation \eqref{e-operations1}. The function $v(x)$ therefore generates the same operations as $\widetilde{w}_2(x)$ at the endpoints and trivially defines the same boundary triple. This is surprising mostly because the function $v(x)$ is not a solution of 
\begin{align*}
(\ell_{\al,\beta}-\lambda_0)f=0
\end{align*}
for any $\lambda_0\in\RR$, as might be expected from Definition \ref{d-quasideriv}. Indeed, the behavior of the functions $\widetilde{w}_1(x)$ and $v(x)$ was essentially determined in \cite{FL}. There, the asymptotic behavior of functions in the maximal domain near the endpoints was described and the two options are indicated by these two functions. 

Alternatively, the function $v(x)$ can be thought of as the truncation of the power series solution $\widetilde{w}_2(x)$ to the first term, and then multiplied by some constant. Importantly, functions like $v(x)$ generating quasi-derivatives will offer much more flexibility when powers of the Jacobi operator are considered in Section \ref{s-powers}. Calculations involving the full solution $w_2(\la,x)$ in the simple sesquilinear form for the uncomposed Jacobi operator were feasible, but it quickly becomes too difficult to get explicit answers for powers.

While $v(x)$ is not a solution, this is offset by the fact that $v(x)$ does have the property that $1/pv^2$ is integrable at each endpoint; a key property in Definition \ref{d-principal} of non-principal solutions. Another important fact is still able to be derived due to this property.

\begin{theo}\label{t-dividev}
Let $f\in\cD\ti{max}^{(\al,\beta)}$. Then 
\begin{align*}
    &\lim_{x\to -1^+}\dfrac{f(x)}{v(x)}=\lim_{x\to -1^+}\beta 2^{\al+1}(1+x)^{\beta}f(x), \text{ and} \\
    &\lim_{x\to 1^-}\dfrac{f(x)}{v(x)}=\lim_{x\to 1^-}\al 2^{\beta+1}(1-x)^{\al}f(x)
\end{align*}
exist and are finite.
\end{theo}

\begin{proof}
This is simply the analog of \cite[Theorem 6.10.9 (i)]{BdS} when $v$ is not a solution. A slight modification to \cite[Lemma 6.11.3]{BdS} and an application of \cite[Lemma 6.10.1]{BdS} shows that the the hypotheses of \cite[Theorem 6.10.9 (i)]{BdS} still hold.
\end{proof}

Although this could have been concluded by analyzing $\widetilde{w}_2(x)$, not having to use solutions here gives hope for a higher power analog with other quasi-derivatives.
This final characterization of $f/v$ allows for a new set of boundary conditions in a boundary triple.

\begin{theo}
Let $f\in\cD\ti{max}^{(\al,\beta)}$. Then $\{\CC^2,\Gamma_0,\Gamma_1\}$, where
    \begin{align}
    \Gamma_0f:=\left(\begin{array}{c}
f^{*}(-1)   \\
f^{*}(1)
\end{array} \right), \hspace{.5cm}
    \Gamma_1f:=\left( \begin{array}{c}
\lim_{x\to -1^+}\frac{f(x)}{v(x)}  \\
\lim_{x\to 1^-}-\frac{f(x)}{v(x)}
\end{array} \right), \hspace{.5cm}
f\in\cD\ti{max}^{(\al,\beta)},
\end{align}
is a boundary triple for $\cD\ti{max}^{(\al,\beta)}$. Explicitly, we have
\begin{align*}
\Gamma_0f:=\left(\begin{array}{c}
\lim_{x\to -1^+}-f(x)-\dfrac{(1+x)f'(x)}{\beta}  \\
\lim_{x\to 1^-}f(x)-\dfrac{(1-x)f'(x)}{\al}
\end{array} \right), \hspace{.3cm}
    \Gamma_1f:=\left( \begin{array}{c}
\lim_{x\to -1^+}\beta 2^{\al+1}(1+x)^{\beta}f(x)   \\
\lim_{x\to 1^-}-\al 2^{\beta+1}(1-x)^{\al}f(x)
\end{array} \right).
\end{align*}
\end{theo}

\begin{proof}
Omitted for brevity. It follows by writing out 
\begin{align*}
    [f,g]\big|_{-1}^1= \langle\Gamma_1f,\Gamma_0g\rangle - \langle\Gamma_0f,\Gamma_1g\rangle,
\end{align*}
and plugging constant multiples of $1$, $(1-x)^{-\al}$ and $(1+x)^{-\beta}$ into the maps to show surjectivity.
\end{proof}

Other differential operators that are bounded from below and in the limit-circle case at an endpoint can benefit from this analysis. A simple example is the classical Legendre differential expression, which is a special case of the Jacobi expression when $\al=\beta=0$, defined as 
\begin{align*}
\ell[f](x)=-((1-x^2)f'(x))',
\end{align*}
on the maximal domain
\begin{align*}
\cD\ti{max} &= \{f:(-1,1)\to \CC
\text{ }:\text{ } f, f'\in\text{AC}\ti{loc}(-1,1); f,\ell[f]\in L^2(-1,1)
\}.
\end{align*}
This maximal domain defines the associated minimal domain given in Definition \ref{d-min}, and the defect indices are $(2,2)$, with both endpoints in the limit-circle case. The Weyl $m$-function for this operator is obtained explicitly in the recent manuscript \cite{FLN} via a different method. The operator was also analyzed in \cite{FL} and the types of asymptotic behavior of functions in $\cD\ti{max}$ at the endpoints was described. Hence, the operations can again be generated by simple functions. Define a $C^2$ function:
\begin{align*}
    v(x):=\left.\begin{cases}
    \dfrac{\ln{(1+x)}}{2} & \text{ for }x\text{ near }-1 \\
    \dfrac{\ln{(1-x)}}{2} & \text{ for }x\text{ near }1
    \end{cases}\right\},
\end{align*}
so that another operation can be defined as $f^{*}(x):=[f,v](x)$. It can be shown that a boundary triple for $\cD\ti{max}$ can be defined using this operation and $f^{[1]}(x)$ from equation \eqref{e-operations1}. However, in this case the function $v(x)$ is essentially $Q_0(x)$, the first Legendre function of the second kind.

A second example is the classical Laguerre differential expression given by
\begin{align*}
    L_{\al}[f](x)=-\dfrac{1}{x^{\al}e^{-x}}\left[x^{\al+1}e^{-x}f'(x)\right]',
\end{align*}
on the maximal domain 
\begin{align*}
    \cD\ti{max}^{\al} &= \left\{f:(0,\infty)\to \CC
\text{ }:\text{ } f, f'\in\text{AC}\ti{loc}(0,\infty); f,\ell[f]\in L^2_{\al}(0,\infty)\right\}
\end{align*}
where $\al>-1$, $\al^2\neq 1/2$ and the Hilbert space $L^2_{\al}(0,\infty):=L^2\left[(0,\infty),x^{\al}e^{-x}\right]$. This maximal domain defines the associated minimal domain given in Definition \ref{d-min}, and the defect indices are $(1,1)$, with $0$ being in the limit-circle case and $\infty$ in the limit-point case. Again, an explicit Weyl $m$-function is given in \cite{FLN} and remarks are made about asymptotic behavior of functions in \cite[Remark 4.15]{FL}. The two $C^2$ functions
\begin{align*}
    u(x):=\left.\begin{cases}
    -1 & \text{ for }x\text{ near }0 \\
    0 & \text{ otherwise}
    \end{cases}\right\}, ~~
    v(x):=\left.\begin{cases}
    x^{-\al} & \text{ for }x\text{ near }0 \\
    0 & \text{ otherwise}
    \end{cases}\right\},
\end{align*}
can define operations $f^{[0]}(0)=[f,v](0)$ and $f^{[1]}(0)=[f,u](0)$ for $f\in\cD\ti{max}^{\al}$. It can be easily shown that such operations generate a boundary triple for $\cD\ti{max}^{\al}$ and the Weyl function then ascertained via an analog of Proposition \ref{t-btlcgeneral}.

%%%%%%%%%%%%%%%%%%%%%%%%%%%%%
%%%%%%%%%%%%%%%%%%%%%%%%%%%%%
\section{Powers of the Jacobi Differential Operator}\label{s-powers}
%%%%%%%%%%%%%%%%%%%%%%%%%%%%%
%%%%%%%%%%%%%%%%%%%%%%%%%%%%%

Let $0<\al,\beta<1$ and $n\in\NN$. It is known \cite{EKLWY} that the $n$th power of the Jacobi differential expression \eqref{e-jacobi}, defined as composing the expression with itself $n$ times, can be expressed in Lagrangian symmetric form as
\begin{align}\label{e-njacobi}
\ell_{{\bf J}}^n[f](x)=-\dfrac{1}{(1-x)^{\al}(1+x)^{\beta}}\sum_{k=1}^n (-1)^k[C(n,k,\al,\beta)(1-x)^{\al+k}(1+x)^{\beta+k}f^{(k)}(x)]^{(k)},
\end{align}
on the maximal domain
\begin{align*}
\cD^{{\bf J}, n}\ti{max}=\{ f\in L^2_{\al,\beta}(-1,1) ~|~ f,f',\dots,f^{(2n-1)}\in AC\ti{loc}(-1,1);\ell_{\al,\beta}^n[f]\in L^2_{\al,\beta}(-1,1)\},
\end{align*}
where the Hilbert space $L^2_{\al,\beta}(-1,1)=L^2\left[(-1,1),(1-x)^{\al}(1+x)^{\beta}\right]$. This maximal domain defines the associated minimal domain given in Definition \ref{d-min}, and the defect indices are $(2n,2n)$. To make the notation more accessible, we are suppressing the dependence on $\al$ and $\beta$ in the definition of $\cD^{{\bf J}, n}\ti{max}$, $\cD^{{\bf J}, n}\ti{min}$ and the defect spaces $\cD_+^{{\bf J}, n}$, $\cD_-^{{\bf J}, n}$, see equation \eqref{e-vN}. Explicit values for the constants $C(n,k,\al,\beta)$ can also be found in \cite{EKLWY}.

The associated sesquilinear form is defined, for $f,g\in\cD^{{\bf J},n}\ti{max}$, via equation \eqref{e-greens}. It can be written explicitly \cite[Section 6]{FFL} as
\begin{align}\label{e-njacobisesqui}
[f,g]_n(x):=\sum_{k=1}^n\sum_{j=1}^k(-1)^{k+j}\bigg\{\big[a_k(x)\overline{g}^{(k)}(x)\big]^{(k-j)}&f^{(j-1)}(x)-\\
&\big[a_k(x)f^{(k)}(x)\big]^{(k-j)}\overline{g}^{(j-1)}(x)\bigg\}\nonumber,
\end{align}
where $a_k(x)=(1-x)^{\al+k}(1+x)^{\beta+k}$. 
 
 We now recall a definition of general quasi-derivatives by Naimark \cite[Section 15.2]{N} that will serve as an abstraction of Definition \ref{d-quasideriv}. Within this section, we will use this more general definition and hope this will not cause any confusion.
 \begin{defn}\label{d-quasiderivgeneral}
Let the function $f$ be associated with the differential expression
 \begin{align*}
     \ell[f]=\sum_{k=0}^n(-1)^k\left[a_ky^{(k)}\right]^{(k)},
 \end{align*}
 where $a_0,\dots,a_n$ are real, sufficiently often differentiable coefficients. The quasi-derivatives of $f$ are defined by the formulas
 \begin{equation}\label{e-quasiderivgeneral}
     \begin{aligned}
     f^{[k]}&=f^{(k)}, \text{ for } k=1,\dots,(n-1); \\
     f^{[n]}&=a_nf^{(n)}, \\
     f^{[n+k]}&=a_{n-k}f^{(n-k)}-\left[f^{[n+k-1]}\right]', \text{ for } k=1,\dots,(n-1); \\
     f^{[2n-1]}&=a_1f'-\left[f^{[2n-2]}\right]'.
     \end{aligned}
 \end{equation}
 For convenience, we also write $f^{[0]}=f$. 
 \end{defn}
 Note that this definition immediately applies to the operator $\ell^n_{{\bf J}}$ given by equation \eqref{e-njacobi} with $a_k(x)=(1-x)^{\al+k}(1+x)^{\beta+k}$ for $k=1,\dots,n$. The main advantage of this general definition of quasi-derivatives is that they create the sesquilinear form a priori. For $f,g\in\cD\ti{max}^{{\bf J},n}$, we have 
 \begin{align}\label{e-sesquidecomp}
     \langle \ell^n_{{\bf J}}[f],g\rangle-\langle f,\ell^n_{{\bf J}}[g]\rangle=[f,g]_n=\sum_{k=1}^n\ \left\{f^{[k-1]}\overline{g}^{[2n-k]}-f^{[2n-k]}\overline{g}^{[k-1]}\right\}.
 \end{align}
Equation \eqref{e-sesquidecomp} will serve as a guide for building a boundary triple for the expression $\ell^n_{{\bf J}}$, and we turn our attention to creating well-defined operations which mimic these quasi-derivatives. This task has no obvious solution, as the lower quasi-derivatives ($k=1,\dots,n$) are clearly not well-defined for all $f\in\cD\ti{max}^{{\bf J},n}$. 

The manuscript \cite{FL} can be of some help here, and we begin by defining four classes of $C^{\infty}(-1,1)$ functions by their boundary asymptotics. Elements of the classes are denoted by $\psi_j^+$, $\psi_j^-$, $\f_j^+$ and $\f_j^-$ for $j\in\NN$:
\begin{equation}\label{e-firstkind}
\begin{aligned}
\f_j^+&:=\left.
\begin{cases}
(1-x)^{j-1}, & \text{ for }x \text{ near }1 \\
0, & \text{ for }x \text{ near }-1
\end{cases}
\right\}, \\
\f_j^-&:=\left.
\begin{cases}
0, & \text{ for }x \text{ near }1 \\
(1+x)^{j-1}, & \text{ for }x \text{ near }-1
\end{cases}
\right\},
\end{aligned}
\end{equation}
\begin{equation}\label{e-secondkinddefn}
\begin{aligned}
\psi_j^+(x)&:=\left.
\begin{cases}
(1-x)^{-\al+j-1}, & \text{ for }x \text{ near }1 \\
0, & \text{ for }x \text{ near }-1
\end{cases}
\right\}, \\
\psi_j^-(x)&:=\left.
\begin{cases}
0, & \text{ for }x \text{ near }1 \\
(1+x)^{-\beta+j-1}, & \text{ for }x \text{ near }-1
\end{cases}
\right\}.
\end{aligned}
\end{equation}
Note that the functions $\f_1^+$ and $\f_1^-$ simply behave like the function $1$ near the endpoints $1$ and $-1$ respectively. The dependence of the functions $\psi_j^+$ and $\psi_j^-$ on the parameters $\al$ and $\beta$ is suppressed here for simplicity. All of these functions are in the maximal domain \cite[Lemma 4.3]{FL} and for $j\leq n$ they are also not in the minimal domain \cite[Theorem 4.4]{FFL}. The functions can also be used to define two finite-dimensional subspaces of $\cD\ti{max}^{{\bf J},n}$:

\begin{align*} 
{\bf D}^n_-:=\spa\left\{\left\{\f_j^-\right\}_{j=1}^{n},\left\{\psi_j^-\right\}_{j=1}^{n}\right\}, ~~{\bf D}^n_+:=\spa\left\{\left\{\f_j^+\right\}_{j=1}^{n},\left\{\psi_j^+\right\}_{j=1}^{n}\right\}.
\end{align*}

\begin{cor}[{\cite[Corollary 4.8]{FL}}]\label{c-defectspacesbasis}
The defect spaces $\cD_+^{{\bf J}, n}\dotplus\cD_-^{{\bf J}, n}={\bf D}^n_-\dotplus{\bf D}^n_+$.
\end{cor}

Regularity properties of functions in the maximal domain can also be shown using these functions. These properties will be useful in some of the proofs later on. 

\begin{theo}\label{t-maxdomaindecomp}\cite[Theorem 4.1]{FL}
If $f\in\cD\ti{max}^{{\bf J},n}$, then for $j=0,\dots,n$ 
\begin{align}\label{e-decompresult1}
    \lim_{x\to \pm 1^{\mp}}(1-x)^{\al+j}(1+x)^{\beta+j}f^{(j)}(x) \text{ is finite}.
\end{align}
Furthermore, for $k=0,\dots,n$ and each $j\in\NN$ such that $j<k$,
\begin{align}\label{e-decompresult2}
\lim_{x\to \pm 1^{\mp}}\left[(1-x)^{\al+k}(1+x)^{\beta+k}f^{(k)}(x)\right]^{(k-j)} \text{ is finite}.
\end{align}
\end{theo}

The convenient basis for the defect spaces given above should allow for operations to be defined that recreate quasi-derivatives. Indeed, this is the case for the functions $\f_j^+$ and $\f_j^-$.

\begin{lem}\label{l-gamma0correct}
The $(2n-s)$ quasi-derivative of $f\in\cD\ti{max}^{{\bf J},n}$ at $x=1$ and $x=-1$ is generated by $\f_s^+$ and $\f_s^-$, respectively, for $s=1,\dots,n$. Explicitly, we have
\begin{equation}\label{e-blockgenerate}
\begin{aligned}
f^{[2n-s]}(1)&=\frac{(-1)^s}{(s-1)!}[f,\f_s^+]_n(1) \\
-f^{[2n-s]}(-1)&=\frac{1}{(s-1)!}[f,\f_s^-]_n(-1).
\end{aligned}
\end{equation}
\end{lem}

\begin{proof}
We will prove the result for the endpoint $x=1$ and the analogous result for $x=-1$ will follow. Consider $(1-x)^{s-1}$, for fixed $s\in\NN$ and $s\leq n$. Deconstruct the expression for the sesquilinear form given by \eqref{e-njacobisesqui} into the terms
\begin{align*}
P(x):=\sum_{k=1}^n\sum_{j=1}^k(-1)^{k+j}\left[(1-x)^{\al+k}(1+x)^{\beta+k}\left[(1-x)^{s-1}\right]^{(k)}\right]^{(k-j)}f^{(j-1)}(x), \\
N(x):=\sum_{k=1}^n\sum_{j=1}^k(-1)^{k+j+1}\left[(1-x)^{\al+k}(1+x)^{\beta+k}f^{(k)}(x)\right]^{(k-j)}\left[(1-x)^{s-1}\right]^{(j-1)},
\end{align*}
so that $[f,(1-x)^{s-1}]_n=\lim\ci{x\to 1^-}[P(x)+N(x)]$. We first analyze $\lim\ci{x\to 1^-}P(x)$ and notice that any terms with $k>s-1$ are automatically 0 for all $j$. Therefore, fix $k$ such that $k\leq s-1$ and calculate
\begin{align*}
    \lim_{x\to 1^-}P(x)&\approx\lim_{x\to 1^-}\sum_{k=1}^n\sum_{j=1}^k(-1)^{k+j}\left[(1+x)^{\beta+k}(1-x)^{\al+k}(1-x)^{s-1-k}\right]^{(k-j)}f^{(j-1)}(x) \\
    &\approx\lim_{x\to 1^-}\sum_{k=1}^n\sum_{j=1}^k\sum_{i=0}^{k-j}(1+x)^{\beta+k-i}(1-x)^{\al+s-1-k+j+i}f^{(j-1)}(x).
\end{align*}
For each $j$, we see that the factor in the sum is 0 as $x\to 1^-$ because of equation \eqref{e-decompresult1} and $s-1-k+j+i>j-1$. We conclude that $P(x)=0$ as $x\to 1^-$.

The expression $N(x)$ is clearly 0 for all $j>s$. Similarly, if $j<s$ then a factor of $(1-x)$ survives differentiation and equation \eqref{e-decompresult2} implies that the entire product has a limit of 0 as $x\to 1^-$. We therefore limit our attention to the case $j=s$, which only occurs when $k\geq s$. This leaves us with
\begin{align}\label{e-matching}
    \lim_{x\to 1^-}N(x)&=\lim_{x\to 1^-}\sum_{k=s}^n(-1)^{k+s+1}\left[(1-x)^{\al+k}(1+x)^{\beta+k}f^{(k)}(x)\right]^{(k-s)}(-1)^{s-1} (s-1)! \nonumber \\
    &=\lim_{x\to 1^-}\sum_{k=s}^n(-1)^k(s-1)!\left[(1-x)^{\al+k}(1+x)^{\beta+k}f^{(k)}(x)\right]^{(k-s)}.
\end{align}
We see from equation \eqref{e-quasiderivgeneral} that we can rewrite
\begin{equation}\label{e-rewritelastn}
\begin{aligned}
f^{[2n-s]}(\pm 1)&=\lim_{x\to \pm 1^{\mp}}\sum_{k=s}^n(-1)^k\left[a_k(x)f^{(k)}(x)\right]^{(k-s)} \text{ for }s\text{ even},\\
f^{[2n-s]}(\pm 1)&=\lim_{x\to \pm 1^{\mp}}\sum_{k=s}^n(-1)^{k+1}\left[a_k(x)f^{(k)}(x)\right]^{(k-s)} \text{ for }s\text{ odd}.
\end{aligned}
\end{equation}
A comparison of equations \ref{e-matching} and \ref{e-rewritelastn} yields the result for $f^{[2n-s]}(1)$.
\end{proof}

Note that the Lemma holds only at the endpoints, and not for $x$ in the interior of the interval. These operations are particularly important because they are 0 when applied to the functions $\f_j^+$ and $\f_j^-$ for other values of $j$.

\begin{cor}\label{c-buildingblocks}
Let $j,k\in\NN$ such that $j,k<n$. Then
\begin{align*}
    \left[\f_j^+\right]^{[2n-k]}(1)=\left[\f_j^-\right]^{[2n-k]}(-1)=0.
\end{align*}
\end{cor}

\begin{proof}
The result follows immediately from Lemma 4.6 of \cite{FL}, which says that two $\f_j^{\pm}$'s against each other in the sesquilinear form is 0.
\end{proof}

%%%%%%%%%%%%%%%%%%%%%%%%%%%%%
%%%%%%%%%%%%%%%%%%%%%%%%%%%%%
\subsection{Regularizations of Quasi-Derivatives}\label{ss-regularizations}
%%%%%%%%%%%%%%%%%%%%%%%%%%%%%
%%%%%%%%%%%%%%%%%%%%%%%%%%%%%

The quasi-derivatives $f^{[j]}(x)$, for $j\in\NN_0$ and $j<n$, are not well-defined for all $f\in\cD\ti{max}^{{\bf J},n}$, so it is necessary to generate operations that represent regularizations of these quasi-derivatives which are well-defined. Unfortunately, the functions $\psi_j^+$ and $\psi_j^-$, for $j\in\NN$ and $j<n$, seem to be unsuitable for this purpose due to some built-in degeneracy. This is discussed in more detail after Corollary \ref{c-newerdefectspaces}, as the main problem is only identifiable after some other structure is introduced. Instead, begin by renumbering the elements so that, for $k=1,\dots,n$, yields
\begin{equation*}
\begin{aligned}
v_k^+(x)&:=\left.
\begin{cases}
(1-x)^{-\al+n-k}, & \text{ for }x \text{ near }1 \\
0, & \text{ for }x \text{ near }-1
\end{cases}
\right\}, \\
v_k^-(x)&:=\left.
\begin{cases}
0, & \text{ for }x \text{ near }1 \\
(1+x)^{-\beta+n-k}, & \text{ for }x \text{ near }-1
\end{cases}
\right\}.
\end{aligned}
\end{equation*}
We now carry out a modified version of the Gram--Schmidt procedure on each set of functions $\left\{v_k^+\right\}_{k=1}^n$ and $\left\{v_k^-\right\}_{k=1}^n$ in order to ensure that 
\begin{align}\label{e-gsgoal}
[\f_j^+,v_k^+]_n(1)=[\f_j^-,v_k^-]_n(-1)=0,
\end{align}
for all $j\neq k$. The new sets of functions will be denoted by $\left\{u_k^+\right\}_{k=1}^n$ and $\left\{u_k^-\right\}_{k=1}^n$. The usual Gram--Schmidt procedure with inner products replaced by sesquilinear forms will not suffice here because any real-valued function set against itself in the sesquilinear form yields 0. The procedure for the construction of the functions $\left\{u_k^+\right\}_{k=1}^n$ is now described and the set $\left\{u_k^-\right\}_{k=1}^n$ will be constructed analogously. Define
\begin{equation}\label{e-gs}
\begin{aligned}
    u_1^+&:=v_1^+/[\f_1^+,v_1^+](1), \\
    u_2^+&:=\left\{v_2^+-\dfrac{[\f_1^+,v_2^+](1)}{[\f_1,u_1^+](1)}u_1^+\right\}/[\f_2^+,v_2^+](1), \\
        &\vdots \\
    u_k^+&:=\left\{v_k^+-\sum_{j=1}^{k-1}\dfrac{[\f_j^+,v_k^+](1)}{[\f_j^+,u_j^+](1)}u_j^+\right\}/[\f_k^+,v_k^+](1),
\end{aligned}
\end{equation}
and the subscript $n$ is suppressed on all of the sesquilinear forms for the sake of simplicity. This convention will continue to be used when exploring some of the consequences of this construction. First, recall a result from \cite{FL} formatted to match the current notation.

\begin{lem}\label{l-overnv}
Let $s,t\in\NN$ such that $s,t\leq n$. If $s>t$, then
\begin{align*}
    \left[\f_s^+,v_t^+\right]_n(1)=\left[\f_s^-,v_t^-\right]_n(-1)=0.
\end{align*}
\end{lem}

\begin{proof}
The exponent of $(1\pm x)$ for the function $\f_s^{\pm}$ is $s-1$ and $n-t$ for the function $v_t^{\pm}$, not counting the $-\al$ or $-\beta$. Their sum is thus $n-1+s-t$. Lemma 4.5 of \cite{FL} says that if the sum of these exponents is greater than $n-1$, equivalently $s>t$ in this case, then the sesquilinear form will yield a value of 0. The result follows.
\end{proof}

The Lemma can be easily modified to work for the new functions $\left\{u_k^+\right\}_{k=1}^n$ and $\left\{u_k^-\right\}_{k=1}^n$. 

\begin{cor}\label{c-overnu}
Let $s,t\in\NN$ such that $s,t\leq n$. If $s>t$, then
\begin{align*}
    \left[\f_s^+,u_t^+\right]_n(1)=\left[\f_s^-,u_t^-\right]_n(-1)=0.
    \end{align*}
\end{cor}

\begin{proof}
The function $u_t^{\pm}$ is constructed by adding and subtracting finitely many $v_i^{\pm}$ functions, for $i\leq t$, in certain proportions. Hence, if $s>t$ then $s>i$ for each $v_i^{\pm}$ and Lemma \ref{l-overnv} proves the result.
\end{proof}

Finally, it is possible to show that the modified Gram--Schmidt procedure has produced functions which satisfy the analog of equation \eqref{e-gsgoal}.

\begin{theo}\label{t-uinteract}
Let $j,k\in\NN$ and $j,k\leq n$. Then for all $j\neq k$ we have
\begin{align}\label{e-claim1}
    [\f_j^+,u_k^+]_n(1)=[\f_j^-,u_k^-]_n(-1)=0.
\end{align}
Additionally,
\begin{align}\label{e-claim2}
[\f_k^+,u_k^+]_n(1)=[\f_k^-,u_k^-]_n(-1)=1.
\end{align}
\end{theo}

\begin{proof}
We prove the result for the endpoint $x=1$ and $x=-1$ will follow analogously. Proceed by induction on $k$. Let $j\in\NN$. Equation \eqref{e-claim1} holds for the base case of 
\begin{align}\label{e-basecase}
    [\f_j^+,u_1^+](1)=0,
\end{align}
for $j>1$ by Lemma \ref{l-overnv}, and equation \eqref{e-claim2} clearly holds when $j=1$. 

Make the inductive hypothesis that the two claims in equations \eqref{e-claim1} and \eqref{e-claim2} hold if $k\leq i$ and consider the case where $k=i+1$:
\begin{align}\label{e-breakdown}
    [\f_j^+,u_{i+1}^+](1)=\left\{[\f_j^+,v_{i+1}^+](1)-\sum_{l=1}^i\dfrac{[\f_l^+,v_{i+1}^+](1)}{[\f_l^+,u_l^+](1)}[\f_j^+,u_l^+](1)\right\}/[\f_{i+1}^+,v_{i+1}^+](1).
\end{align}
If $j>i+1$ then equation \eqref{e-claim1} holds by Corollary \ref{c-overnu}. Let $j<i+1$. Terms when $l<j$ in equation \eqref{e-breakdown} are then also zero by Corollary \ref{c-overnu}. But terms when $l>j$ are also zero by the inductive hypothesis. The one remaining term in the sum, when $l=j$, is then cancelled by the term in front. Hence, equation \eqref{e-claim1} has been proven when $k=i+1$. 

Let $j=i+1$. Then each term in the sum is zero by Corollary \ref{c-overnu} and equation \eqref{e-claim2} is immediately shown when $k=i+1$. The Theorem then follows at the endpoint $x=1$ by the principle of mathematical induction for $k\in\NN$ and $k\leq n$. 
\end{proof}

We now see that the denominators in the modified Gram--Schmidt process are equal to $1$ by construction. However, the intended cancellation of operations is more clear without simplifying these terms, so they will continue to be expressed in the form of equation \eqref{e-gs}. 

The theorem determines the interaction between $u$'s and $\f$'s in the sesquilinear form. Corollary \ref{c-buildingblocks} showed how the the functions $\f$ behave against each other, so it remains only to analyze the behavior of the new $u$ functions. 

\begin{lem}\label{l-vv}
Let $s,t\in\NN$ such that $s,t\leq n$. Then 
\begin{align*}
    [v_s^+,v_t^+]_n(1)=[v_s^-,v_t^-]_n(-1)=0.
\end{align*}
\end{lem}

\begin{proof}
We prove the result for the endpoint $x=1$ and $x=-1$ will follow analogously. Let $s,t\in\NN$, $s,t\leq n$ and without loss of generality assume that $s\neq t$. Deconstruct the expression for the sesquilinear form given by \eqref{e-njacobisesqui} into the terms
\begin{align*}
P(x):=\sum_{k=1}^n\sum_{j=1}^k\left\{a_k(x)\left[(1-x)^{-\al+n-t}\right]^{(k)}\right\}^{(k-j)}\left[(1-x)^{-\al+n-s}\right]^{(j-1)}, \nonumber\\
N(x):=\sum_{k=1}^n\sum_{j=1}^k(-1)^{k+j+1}\left\{a_k(x)\left[(1-x)^{-\al+n-s}\right]^{(k)}\right\}^{(k-j)}\left[(1-x)^{-\al+n-t}\right]^{(j-1)},
\end{align*}
so that $[v_s^+,v_t^+]_n(1)=\lim\ci{x\to 1^-}[P(x)+N(x)]$. We first analyze $\lim\ci{x\to 1^-}P(x)$ and calculate
\begin{align*}
    \lim_{x\to 1^-}P(x)&\approx \lim_{x\to 1^-}\sum_{k=1}^n\sum_{j=1}^k(-1)^{k+j}\left[(1+x)^{\beta+k}(1-x)^{n-t}\right]^{(k-j)}(1-x)^{-\al+n-s-j+1}(x) \nonumber\\
    &\approx \lim_{x\to 1^-}\sum_{k=1}^n\sum_{j=1}^k\sum_{i=0}^{k-j}(1+x)^{\beta+k-i}(1-x)^{-\al+2n-t-s+1-k+i},
\end{align*}
when $i\leq n-t$, otherwise the result is 0. If the exponent $-\al+2n-k-t-s+i+1$ is positive for a combination of $k,i$ then the limit of such a term is clearly 0. If the exponent is negative, group together like terms and observe that the minimum possible exponent occurs when $k=n$, $s=n$, $t=n-1$ and $i=0$, and is $-\al-n+2$. This allows for the decomposition
\begin{align}\label{e-lowpowers}
     \lim_{x\to 1^-}P(x)\approx\lim_{x\to 1^-}\sum_{l=0}^{n-2} h_l(x)(1-x)^{-\al-l},
\end{align}
with functions $h_l(x)$ that are constants times powers of $(1+x)$, but may be identically 0 (depending on the choice of $s$ and $t$) . A similar analysis clearly holds for $N(x)$, with the condition that $i\leq n-s$ naturally arising. The analog of equation \eqref{e-lowpowers} for $N(x)$ can be then be added to equation \eqref{e-lowpowers} so that
\begin{align*}
    [v_s^+,v_t^+]_n(1)=\lim_{x\to 1^-}P(x)+N(x)\approx\lim_{x\to 1^-}\sum_{l=0}^{n-2} \widetilde{h}_l(x)(1-x)^{-\al-l},
\end{align*}
for some functions $\widetilde{h}_l(x)$ that don't go to 0 in the limit unless they are identically 0. Indeed, assume that at least one $\widetilde{h}_l$ function is nonzero. The left hand side of this equation is finite and exists by Theorem \ref{t-limits}, and the right hand side consists of one or more nonzero terms which go to infinity at different rates in the limit. This is a contradiction to Theorem \ref{t-limits}, and thus each $\widetilde{h}_l$ function must be identically 0. The result follows for the endpoint $x=1$.
\end{proof}

\begin{cor}\label{c-uu}
Let $j,k\in\NN$ such that $j,k\leq n$. Then 
\begin{align*}
    [u_j^+,u_k^+]_n(1)=[u_j^-,u_k^-]_n(-1)=0.
\end{align*}
\end{cor}

\begin{proof}
The Corollary follows from Lemma \ref{l-vv} because the function $u_k^{\pm}$ is just a finite linear combination of $v_i^{\pm}$ functions, where $i\leq k$.
\end{proof}

The functions $\left\{u_k^+\right\}_{k=1}^n$ and $\left\{u_k^-\right\}_{k=1}^n$ now produce a much simpler structure for the defect spaces when put into a matrix of sesquilinear forms, similar to that of \cite[Equation 4.22]{FFL}, as compared to the starting families $\left\{\psi_j^+\right\}_{j=1}^{n}$ and $\left\{\psi_j^-\right\}_{j=1}^{n}$. At the endpoint $x=1$, observe
\begin{align}\label{e-veffect}
\sbox0{$\begin{matrix}[\f_1,\f_1] & \dots & [\f_1,\f_n] \\ \vdots & \iddots & \vdots \\ [\f_n,\f_1] & \dots & [\f_n,\f_n]\end{matrix}$}
\sbox1{$\begin{matrix}[\f_1,u_n] & \dots & [\f_1,u_1] \\ \vdots & \iddots & \vdots \\ [\f_n,u_n] & \dots & [\f_n,u_1]\end{matrix}$}
\sbox2{$\begin{matrix}[u_n,\f_1] & \dots & [u_n,\f_n] \\ \vdots & \iddots & \vdots \\ [u_1,\f_1] & \dots & [u_1,\f_n]\end{matrix}$}
\sbox3{$\begin{matrix}[u_n,u_n] & \dots & [u_n,u_1] \\ \vdots & \iddots & \vdots \\ [u_1,u_n] & \dots & [u_1,u_1]\end{matrix}$}
\left(
\begin{array}{c|c}
\usebox{0}&\usebox{1}\vspace{-.35cm}\\\\
\hline\vspace{-.35cm}\\
 \vphantom{\usebox{0}}\usebox{2}&\usebox{3}
\end{array}
\right)
=
\sbox0{$\begin{matrix} &  & \\  & 0 & \\ & & \end{matrix}$}
\sbox1{$\begin{matrix}0 & & 1 \\  & \iddots &  \\ 1 & & 0\end{matrix}$}
\sbox2{$\begin{matrix}0 & & 1 \\  & \iddots &  \\ 1 & & 0\end{matrix}$}
\sbox3{$\begin{matrix} & & \\ & 0 & \\ & & \end{matrix}$}
\left(
\begin{array}{c|c}
\usebox{0}&\usebox{1}\vspace{-.35cm}\\\\
\hline\vspace{-.35cm}\\
 \vphantom{\usebox{0}}\usebox{2}&\usebox{3}
\end{array}
\right),
\end{align}
where each sesquilinear form is evaluated at $1$ and the ``$+$'' superscripts on functions are suppressed for the sake of simplicity. The upper-right and lower-left quadrants were determined by Theorem \ref{t-uinteract}. The upper-left and bottom-right quadrants are results of Corollaries \ref{c-buildingblocks} and \ref{c-uu}, respectively. Define two finite-dimensional subspaces of $\cD\ti{max}^{{\bf J},n}$:

\begin{align*} 
\widetilde{{\bf D}}^n_-:=\spa\left\{\left\{\f_j^-\right\}_{j=1}^{n},\left\{u_j^-\right\}_{j=1}^{n}\right\}, ~~\widetilde{{\bf D}}^n_+:=\spa\left\{\left\{\f_j^+\right\}_{j=1}^{n},\left\{u_j^+\right\}_{j=1}^{n}\right\}.
\end{align*}

\begin{cor}\label{c-newerdefectspaces}
The defect spaces $\cD_+^{{\bf J}, n}\dotplus\cD_-^{{\bf J}, n}=\widetilde{{\bf D}}^n_-\dotplus\widetilde{{\bf D}}^n_+$.
\end{cor}

\begin{proof}
Notice that
\begin{align*}
    \spa\left\{u_j^+\right\}_{j=1}^n=\spa\left\{\psi_j^+\right\}_{j=1}^n \text{  and  }\spa\left\{u_j^-\right\}_{j=1}^n=\spa\left\{\psi_j^-\right\}_{j=1}^n.
\end{align*}
The result thus follows from Corollary \ref{c-defectspacesbasis}.
\end{proof}

Taking a brief aside, it is now possible to discuss why the change from the $\psi_j$ functions to the $u_j$ functions was so essential. The $\psi_j$ functions plugged into the matrix of sesquilinear forms in equation \eqref{e-veffect} instead of the $u_j$'s yields a very different picture. The situation is described in \cite[Theorem 4.7]{FL}, with the key being that the upper-right and lower-left quadrants are merely upper triangular and not diagonal. This interaction between the functions $\psi_j$ and $\f_j$ in the sesquilinear form yields much more complicated operations that are not able to be easily analyzed. Indeed, the operations seem degenerate in some way, making it difficult to isolate the terms responsible for new behavior as $j$ increases.

The structure created by these new functions allows for the definition of operations which will be forged into a boundary triple. First, we investigate how these new functions act in the sesquilinear form.

\begin{lem}\label{l-gamma1correct}
Let $f\in\cD\ti{max}^{{\bf J},n}$. Then, for $j\in\NN$ and $j\leq n$, the representation
\begin{align}\label{e-newrep}
    [f,u_j^{\pm}]_n(\pm 1)=\lim_{x\to \pm 1^{\mp}}(\mp 1)^{j-1}\dfrac{f^{(j-1)}(x)}{(j-1)!}+\left\{\sum_{k=j}^{2n-1} h_{j,k}(x)f^{(k)}(x)\right\},
\end{align}
holds for some functions $h_{j,k}(x)$. 
\end{lem}

\begin{proof}
We prove the result for the endpoint $x=1$ and $x=-1$ will follow analogously. To this end, the $+$ superscript on functions and the $n$ subscript on the sesquilinear form will be suppressed during the proof. Let $j\in\NN$ such that $j\leq n$. It is clear from the definition of the sesquilinear form in equation \eqref{e-njacobisesqui} that the finitely many terms can be written out and rearranged so that
\begin{align}\label{e-utaylor}
    [f,u_j](1)=\lim_{x\to 1^-}\sum_{l=0}^{2n-1}h_{j,l}(x)f^{(l)}(x),
\end{align}
for some functions $h_{j,l}(x)$ which may go to infinity in the limit. 

The claim is then that both $h_{j,l}(x)=0$ for $l<j-1$ and $h_{j,j-1}(x)=(-1)^{j-1}/(j-1)!$. These properties follow from the modified Gram--Schmidt procedure carried out in equation \eqref{e-gs}. Notice that it can be easily shown that the limits of these functions have the desired values using equation \eqref{e-veffect} and induction. However, the Gram--Schmidt procedure from equation \eqref{e-gs} will show the stronger fact that the functions are identically the desired constants. Begin by defining an analogous decomposition to equation \eqref{e-utaylor} for the functions $v_j$:
\begin{align}\label{e-vtaylor}
    [f,v_j](1)=\lim_{x\to 1^-}\sum_{l=0}^{2n-1}g_{j,l}(x)f^{(l)}(x).
\end{align}

Let $j=1$. By equation \eqref{e-gs} we have
\begin{align*}
    [\f_1,u_1](1)=\lim_{x\to 1^-}\left\{\dfrac{[\f_1,v_1](x)}{[\f_1,v_1](x)}\right\}=\lim_{x\to 1^-}\dfrac{g_{1,0}(x)}{g_{1,0}(x)}=\lim_{x\to 1^-}1\cdot{\f_1}^{(0)}(x),
\end{align*}
so that $h_{1,0}(x)=1$. Theorem \ref{t-limits} was used as the definition for the sesquilinear form in the denominator, and this allowed the limit to be pulled outside of the calculations.

Proceed by induction on $j$ to prove the two claims. Let $j=2$ be the base case. Again, using equation \eqref{e-gs} it can be written
\begin{align*}
 [\f_1,u_2](1)=\lim_{x\to 1^-}\left\{\left[g_{2,0}(x)-\dfrac{g_{2,0}(x)h_{1,0}(x)}{h_{1,0}(x)}\right]/g_{2,1}(x)\right\}=\lim_{x\to 1^-}0\cdot{\f_1}^{(0)}(x),
 \end{align*}
so that $h_{2,0}(x)=0$. The fact that $h_{2,1}(x)=-1$ follows analogously by calculating
\begin{align*}
    [\f_2,u_2](1)=\lim_{x\to 1^-}1=\lim_{x\to 1^-}\dfrac{-1}{1!}\cdot\f_2^{(1)}(x).
\end{align*}
Assume the inductive hypothesis that $h_{j_1,l}(x)=0$ for $l<j_1-1$ and $h_{j_1,j_1-1}(x)=1$ for arbitrary $2<j_1<n$. Consider $j=j_1+1$. Then for each $i\in\NN$ such that $i\leq j_1$
\begin{align*}
    [\f_i,u_{j_1+1}](1)&=\lim_{x\to 1^-}\left\{[\f_i,v_{j_1+1}](x)-\sum_{m=1}^{j_1}\dfrac{[\f_m,v_{j_1+1}](x)}{[\f_m,u_m](x)}[\f_i,u_m](x)\right\}/[\f_{j_1+1},v_{j_1+1}](x) \\
    &=\lim_{x\to 1^-}\left\{g_{j_1+1,i-1}-\dfrac{g_{j_1+1,i-1}(x)}{g_{i,i-1}(x)}g_{i,i-1}(x)\right\}/g_{j_1+1,j_1}(x) \\
    &=\lim_{x\to 1^-}0\cdot{\f_i^+}^{(i-1)}(x),
\end{align*}
so that $h_{j_1+1,i-1}(x)=0$. The inductive hypothesis was used in reducing the sum to one term. The fact that $h_{j_1+1,j_1}(x)=(-1)^{j-1}/(j-1)!$ follows analogously by calculating
\begin{align*}
    [\f_j,u_j](1)=\lim_{x\to 1^-}1=\lim_{x\to 1^-}\dfrac{(-1)^{j-1}}{(j-1)!}\cdot\f_j^{(j-1)}(x).
\end{align*}
The principle of mathematical induction then says that, for all $j\in\NN$ such that $j\leq n$, both $h_{j,l}(x)=0$ for $l<j-1$ and $h_{j,j-1}(x)=1$. The Lemma has thus been proven for the endpoint $x=1$.
\end{proof}

As mentioned in the proof, the Lemma is stronger than saying that the functions $h_{j,l}(x)$=0, for $l<j-1$, and $h_{j,j-1}(x)=1$ in their limits. A priori, the functions $h_{j,l}$ going to 0 in the limit when $l<j-1$ is insufficient, as $f^{(l)}(x)$ may go to infinity at a faster rate for some $f\in\cD\ti{max}^{{\bf J},n}$. Corollary \ref{c-newerdefectspaces} and Theorem \ref{t-neumanndecomp} allow any function $f\in\cD\ti{max}^{{\bf J},n}$ to be written as 
\begin{equation}\label{e-newneumanndecomp}
\begin{aligned}
f=f_0&+c_1\f_1^++\dots+c_n\f_n^++c_{n+1}\f_1^-+\dots+c_{2n}\f_n^- \\
&+c_{2n+1}u_1^+\dots+c_{3n}u_n^++c_{3n+1}u_1^-+\dots+c_{4n}u_n^-, \nonumber
\end{aligned}
\end{equation}
for some constants $c_1,\dots,c_{4n}$ that are determined by $f$ and $f_0\in\cD\ti{min}^{{\bf J},n}$. The definition of the minimal domain and equation \eqref{e-veffect} thus say that 
\begin{align*}
    [f,u_j^+]_n(1)=c_j[\f_j^+,u_j^+]_n(1)=c_j.
\end{align*}
Analogs of this reasoning hold for all of the operations in the matrix of equation \eqref{e-veffect}. 

Lemma \ref{l-gamma1correct} defines some regularizations of quasi-derivatives but these are not assumed to be unique, just as the boundary triple in the following Subsection will not be uniquely to the naturally generated self-adjoint extension. Unfortunately, it is unknown if anything can be said of the remaining functions $h_{j,k}(x)$ in Lemma \ref{l-gamma1correct}.

%%%%%%%%%%%%%%%%%%%%%%%%%%%%%
%%%%%%%%%%%%%%%%%%%%%%%%%%%%%
\subsection{A Natural Boundary Triple}\label{ss-natural}
%%%%%%%%%%%%%%%%%%%%%%%%%%%%%
%%%%%%%%%%%%%%%%%%%%%%%%%%%%%

A representation like that of Lemma \ref{l-gamma1correct} allows for the explicit construction of a boundary triple using quasi-derivatives as a guide, via equation \eqref{e-sesquidecomp}. In order for the new operations to match Definition \ref{d-quasiderivgeneral}, some slight modifications are needed. For $f\in\cD\ti{max}^{{\bf J},n}$ and $j\in\NN$ such that $j\leq n$, define the operations
\begin{equation}\label{e-new1ops}
\begin{aligned}
f^{\{j-1\}}(1)&:=(-1)^{j-1}(j-1)!\cdot[f,u_j^+]_n(1), \\
f^{\{j-1\}}(-1)&:=(j-1)!\cdot[f,u_j^-]_n(-1).
\end{aligned}
\end{equation}

Finally, we can define the maps $\Gamma_0,\Gamma_1:\cD\ti{max}^{{\bf J},n}\to \CC^{2n}$ via
\begin{equation}\label{e-newbt}
\begin{aligned}
\Gamma_0 f:=\left(
\begin{array}{c}
-f^{[n]}(-1) \\
\vdots \\
-f^{[2n-1]}(-1) \\
f^{[n]}(1) \\
\vdots \\
f^{[2n-1]}(1)
\end{array} 
\right), \hspace{.2cm}
\Gamma_1 f:=\left(
\begin{array}{c}
f^{\{n-1\}}(-1) \\
\vdots \\
f^{\{0\}}(-1) \\
f^{\{n-1\}}(1) \\
\vdots \\
f^{\{0\}}(1)
\end{array} 
\right),
\end{aligned}
\end{equation}

where the quasi-derivatives in the definition of $\Gamma_0$ are given by Lemma \ref{l-gamma0correct}.

\begin{theo}\label{t-newbt}
Let $\Gamma_0$ and $\Gamma_1$ be given by equation \eqref{e-newbt}. Then $\{\CC^{2n},\Gamma_0,\Gamma_1\}$ is a boundary triple for  $\cD\ti{max}^{{\bf J},n}$.
\end{theo}

\begin{proof}
Let $f,g\in\cD\ti{max}^{{\bf J},n}$. We aim to show that 
\begin{align*}
    [f,g]_n\bigg|_{-1}^1=[f,g]_n(1)-[f,g]_n(-1)=\langle\Gamma_1f,\Gamma_0g\rangle-\langle\Gamma_0f,\Gamma_1g\rangle.
\end{align*}
Lemma \ref{l-gamma1correct} and equation \eqref{e-new1ops} yield, for $j\in\NN$ such that $j\leq n$,
\begin{equation}\label{e-part1proofsetup}
\begin{aligned}
f^{\{j-1\}}(1)&=(-1)^{j-1}(j-1)!\cdot[f,u_j^+]_n(1)=\lim_{x\to \pm 1^-}f^{[j-1]}(x)+\left\{\sum_{l=j}^{2n-1} h_{j,l}(x)f^{(l)}(x)\right\}, \\
f^{\{j-1\}}(-1)&=(j-1)!\cdot[f,u_j^-]_n(1)=\lim_{x\to \pm -1^+}f^{[j-1]}(x)+\left\{\sum_{l=j}^{2n-1} \widetilde{h}_{j,l}(x)f^{(l)}(x)\right\},
\end{aligned}
\end{equation}
for some functions $h_{j,l}(x)$ and $\widetilde{h}_{j,l}(x)$. The inner product in $\CC^{2n}$ is a linear operator, so begin by considering only the quasi-derivative term in $\Gamma_1$. With this truncated $\Gamma_1$, it is clear that 
\begin{align*}
    \langle\Gamma_1f,\Gamma_0g\rangle-\langle\Gamma_0f,\Gamma_1g\rangle=\sum_{k=1}^n\left\{f^{[k-1]}\overline{g}^{[2n-k]}-f^{[2n-k]}\overline{g}^{[k-1]}\right\}.
\end{align*}
Equation \eqref{e-sesquidecomp} says that the sesquilinear form evaluated at $x=1$ is created, and the minus signs in all of the terms at $x=-1$ in $\Gamma_0$ ensure that $-[f,g]_n(-1)$ is also generated. Now consider $\Gamma_1$ to be only the extra summation from equation \eqref{e-part1proofsetup}. We analyze the endpoint $x=1$ and our conclusions will hold analogously at $x=-1$. First, notice that if $f\in\cD\ti{max}^{{\bf J},n}$ is such that $f^{\{j-1\}}(1)=0$ for some $j\in\NN$ with $j\leq n$, then
\begin{equation}\label{e-getting0}
\begin{aligned}
    0=&\lim_{x\to \pm 1^-}f^{[j-1]}(x)+\left\{\sum_{l=j}^{2n-1} h_{j,l}(x)f^{(l)}(x)\right\} \text{   implies } \\
    &\lim_{x\to \pm 1^-}-f^{[j-1]}(x)=\lim_{x\to \pm 1^-}\sum_{l=j}^{2n-1} h_{j,l}(x)f^{(l)}(x).
\end{aligned}
\end{equation}
The boundary triple with the truncated $\Gamma_1$ now yields
\begin{equation}\label{e-everythingelse}
\begin{aligned}
    \langle\Gamma_1f,\Gamma_0g\rangle-\langle\Gamma_0f,\Gamma_1g\rangle=\lim_{x\to 1^-}\sum_{k=1}^n & g^{[2n-k]}(x)\left\{\sum_{l=k}^{2n-1} y_{k,l}(x)f^{(l)}(x)\right\}\\
    &-f^{[2n-k]}(x)\left\{\sum_{l=k}^{2n-1} z_{k,l}(x)g^{(l)}(x)\right\},
\end{aligned}
\end{equation}
for some functions $y_{k,l}(x)$ and $z_{k,l}(x)$. 

The claim is that equation \eqref{e-everythingelse} is equal to 0 for all $f,g\in\cD\ti{max}^{{\bf J},n}$. However, equation \eqref{e-newneumanndecomp} implies that it is enough to consider $f$ and $g$ taken from the families $\{\f_i^+\}_{i=1}^n$ and $\{u_i^+\}_{i=1}^n$. If both $f$ and $g$ are constant multiples of functions from $\{u_i^+\}_{i=1}^n$. Then equation \eqref{e-everythingelse} can be simplified by equation \eqref{e-getting0} to
\begin{equation*}
\begin{aligned}
    \langle\Gamma_1f,\Gamma_0g\rangle-\langle\Gamma_0f,\Gamma_1g\rangle&=-\lim_{x\to 1^-}\sum_{k=1}^n g^{[2n-k]}(x)f^{[k-1]}(x)-f^{[2n-k]}(x)g^{[k-1]} \\
    &=-[f,g]_n(1)=0.
\end{aligned}
\end{equation*}

Without loss of generality, consider the case where $f(x)=\f_s^+(x)$ and $g(x)=\f_t^+(x)$, for some $s,t\in\NN$ and $s,t\leq n$. Then
\begin{equation*}
\begin{aligned}
    \lim_{x\to 1^-}\sum_{l=k}^{2n-1} y_{k,l}(x)\left[\f_s^+\right]^{(l)}(x)=
    \begin{cases}
    \lim_{x\to 1^-}-\left[\f_s^+\right]^{[k-1]}(x) & \text{ for }k\neq s, \\
    \lim_{x\to 1^-}(-1)^{s-1}(s-1)!-\left[\f_s^+\right]^{[s-1]} & \text{ for }k=s,
    \end{cases}
\end{aligned}
\end{equation*}
where equation \eqref{e-getting0} can be easily modified to show the case $k=s$. In both cases the limit is clearly $0$. An analog holds for $g(x)$. As the limits of $f^{[2n-k]}(x)$ and $g^{[2n-k]}(x)$ are also zero by Corollary \ref{c-buildingblocks}, we conclude that equation \eqref{e-everythingelse} is 0 for such functions. 

Finally, without loss of generality, consider the case where $f(x)=u_t^+(x)$ and $g(x)=\f_s^+(x)$ for $s\neq t$. Then, using the above simplifications we have
\begin{align*}
    \langle\Gamma_1f,\Gamma_0g\rangle-\langle\Gamma_0f,\Gamma_1g\rangle&=-\lim_{x\to 1^-}\sum_{k=1}^n [u_t^+]^{[2n-k]}(x)[\f_s^+]^{[k-1]}(x)-[\f_s^+]^{[2n-k]}(x)[u_t^+]^{[k-1]}(x) \\
    &-\lim_{x\to 1^-}[u_t^+]^{[2n-s]}(-1)^{s-1}(s-1)!=-[\f_s^+,u_t^+]_n(1)-0=0.
\end{align*}
The case where $s=t$ similarly yields
\begin{align*}
    \langle\Gamma_1f,\Gamma_0g\rangle-\langle\Gamma_0f,\Gamma_1g\rangle&=-[\f_s^+,v_t^+]_n(1)-\lim_{x\to 1^-} \dfrac{(-1)^s}{(s-1)!}(-1)^{s-1}(s-1)!=-1-(-1)=0.
\end{align*}
It is clear that if $f$, $g$, or both, belong to the minimal domain the result is also 0. The claim that equation \eqref{e-everythingelse} is equal to 0 for all $f,g\in\cD\ti{max}^{{\bf J},n}$ has thus been shown, and the desired form at the endpoint $x=1$ follows.

It remains only to show that the mappings $\Gamma_0$ and $\Gamma_1$ are surjective onto $\CC^{2n}$. However, equation \eqref{e-veffect} clearly shows that linear combinations of the family $\{v_j^{\pm}\}_{j=1}^n$ will take all possible values in $\CC^{2n}$ under the map $\Gamma_0$ and linear combinations of the family $\{\f_j^{\pm}\}_{j=1}^n$ will take all possible values in $\CC^{2n}$ under the map $\Gamma_1$. Thus, we conclude that $\{\CC^{2n},\Gamma_0,\Gamma_1\}$ is a boundary triple for $\cD\ti{max}^{{\bf J},n}$.
\end{proof}

It should be noted that there are many possible ways to prove Theorem \ref{t-newbt}, some of which are shorter. The advantage of this proof is it shows how the decomposition from equation \eqref{e-newneumanndecomp} interacts with the extra summations from Lemma \ref{l-gamma1correct}, which is valuable for building intuition.

%%%%%%%%%%%%%%%%%%%%%%%%%%%%%
%%%%%%%%%%%%%%%%%%%%%%%%%%%%%
\section{Weyl m-Functions}\label{s-mfunctions}
%%%%%%%%%%%%%%%%%%%%%%%%%%%%%
%%%%%%%%%%%%%%%%%%%%%%%%%%%%%

The constructed boundary triple in equation \eqref{e-newbt} allows for the determination of explicit Weyl $m$-functions using Subsection \ref{ss-bt}. In particular, four examples will be computed: the two natural self-adjoint extensions which have the kernels of $\Gamma_0$ and $\Gamma_1$ as their domains, separated boundary conditions, and an analog of periodic boundary conditions.

It is first necessary to make some comments about solutions to the differential equation $\ell_{{\bf J}}^n$ given by equation \eqref{e-njacobi}. The deficiency indices of the associated minimal domain are $(2n,2n)$ so given a $\la\in\CC$, there are $2n$ linearly independent solutions to the equation
\begin{align}\label{e-neigenvalue}
    \ell_{{\bf J}}^n[f]=\la f.
\end{align}
However, these solutions can be defined via solutions to the uncomposed equation
\begin{align*}
    \ell_{{\bf J}}[f]=\la_j f,
\end{align*}
where each $\{\la_j\}_{j=1}^n$ is distinct and $\la_j^n=\la$. To each of these associated $n$ equations there are two solutions guaranteed due to the fact that $\ell_{{\bf J}}$ is in the limit-circle case at both endpoints. Denote these two solutions by $f_j$ and $g_j$ and decompose $\la_j=\mu_j(\mu_j+\al+\beta+1)$ so that the equation is in the usual format. Using the change of variables $t=(1-x)/2$, so that $1-t=(1+x)/2$, the solutions to equation \eqref{e-neigenvalue} are written as
\begin{equation}\label{e-nsolution}
\begin{aligned}
f_j(t):=
\left.\begin{cases}
e_jF(-\mu_j,\mu_j+\al+\beta+1;\al+1;t) & \text{ at }t=0\text{ }(x=1) \\
F(-\mu_j,\mu_j+\al+\beta+1;\beta+1;1-t) & \text{ at }t=1\text{ }(x=-1)
\end{cases}\right\}, \\
g_j(t):=
\left.\begin{cases}
t^{-\al}F(-\mu_j-\al,\mu_j+\beta+1;1-\al;t) & \text{ at }t=0\text{ }(x=1) \\
(1-t)^{-\beta}F(-\mu_j-\beta,\mu_j+\al+1;1-\beta;1-t) & \text{ at }t=1\text{ }(x=-1)
\end{cases}\right\},
\end{aligned}
\end{equation}
where the constant $e_j$ normalizes the sesquilinear form so that $[f_j,g_j](1)=1$ for all $j$. A priori, this constant is clearly just $1/[\widetilde{f}_j,g_j](1)$, where $\widetilde{f}_j:=f_j/e_j$, but a more explicit form can be given in some cases, see equation \eqref{e-edefn}. 

All calculations in this section will take place in $\cD\ti{max}^{{\bf J},n}$, so the sesquilinear form will always be $[\cdot,\cdot]_n(x)$ and the $n$ subscript will be omitted throughout for simplicity. We now assume that the fundamental system of solutions $\{f_1,\dots,f_n,g_1,\dots,g_n\}$ to equation \eqref{e-neigenvalue} has the property that the matrix
\begin{align}\label{e-ninitial}
\sbox0{$\begin{matrix}-f_1^{[n]}(-1) & \dots & -f_n^{[n]}(-1) \\ \vdots & \iddots & \vdots \\ -f_1^{[2n-1]}(-1) & \dots & -f_n^{[2n-1]}(-1)\end{matrix}$}
\sbox1{$\begin{matrix}-g_1^{[n]}(-1) & \dots & -g_n^{[n]}(-1) \\ \vdots & \iddots & \vdots \\ -g_1^{[2n-1]}(-1) & \dots & -g_n^{[2n-1]}(-1)\end{matrix}$}
\sbox2{$\begin{matrix}f_1^{\{n-1\}}(-1) & \dots & f_n^{\{n-1\}}(-1) \\ \vdots & \iddots & \vdots \\ f_1^{\{0\}}(-1) & \dots & f_n^{\{0\}}(-1)\end{matrix}$}
\sbox3{$\begin{matrix} g_1^{\{n-1\}}(-1) & \dots & g_n^{\{n-1\}}(-1) \\ \vdots & \iddots & \vdots \\ g_1^{\{0\}}(-1) & \dots & g_n^{\{0\}}(-1)\end{matrix}$}
\left(
\begin{array}{c|c}
\usebox{0}&\usebox{1}\vspace{-.35cm}\\\\
\hline\vspace{-.35cm}\\
 \vphantom{\usebox{0}}\usebox{2}&\usebox{3}
\end{array}
\right)
\end{align}
satisfies initial conditions that make it equal to 
\begin{align*}
\sbox0{$\begin{matrix} &  & \\  & 0 & \\ & & \end{matrix}$}
\sbox1{$\begin{matrix}0 & & 1 \\  & \iddots &  \\ 1 & & 0\end{matrix}$}
\sbox2{$\begin{matrix}0 & & 1 \\  & \iddots &  \\ 1 & & 0\end{matrix}$}
\sbox3{$\begin{matrix} & & \\ & 0 & \\ & & \end{matrix}$}
\left(
\begin{array}{c|c}
\usebox{0}&\usebox{1}\vspace{-.35cm}\\\\
\hline\vspace{-.35cm}\\
 \vphantom{\usebox{0}}\usebox{2}&\usebox{3}
\end{array}
\right).
\end{align*}

Note that these conditions mean that $[f_j,g_j](-1)=1$ for all $j$ automatically. In order to take advantage of our initial conditions, we need to state connection formulas between the endpoints for the hypergeometric function. Namely, \cite[Equation 15.8.4]{DLMF} says
\begin{equation}\label{e-nconnect}
\begin{aligned}
    F(a,b;c;1-z)=&\frac{\pi}{\sin(\pi(c-a-b))\Gamma(c-a)\Gamma(c-b)}F(a,b;a+b-c+1;z) \\
    &-\frac{\pi(1-z)^{c-a-b}}{\sin(\pi(c-a-b))\Gamma(a)\Gamma(b)}F(c-a,c-b;c-a-b+1;z),
\end{aligned}
\end{equation}
which holds for $|ph(z)|<\pi$ and $|ph(1-z)|<\pi$, where $ph(z)$ denotes the principal value of $z$. It is thus necessary to define four sets of parameters that will be dependent on the choice of the solution. For $j\in\NN$ such that $j\leq n$,
\begin{equation}\label{e-nfirstconnect}
\begin{aligned}
   \gamma_j&:=\dfrac{-\pi e_j2^{\beta}}{\sin(-\beta\pi)\Gamma(-\mu_j)\Gamma(\mu_j+\al+\beta+1)}, \\
   \eps_j&:=\dfrac{-\pi 2^{\beta}}{\sin(-\beta\pi)\Gamma(-\mu_j-\al)\Gamma(\mu_j+\beta+1)},
\end{aligned}
\end{equation}
\begin{equation}\label{e-nsecondconnect}
\begin{aligned}
   \delta_j&:=\dfrac{\pi e_j}{\sin(-\beta\pi)\Gamma(\mu_j+\al+1)\Gamma(-\mu_j-\beta)}, \\
   \eta_j&:=\dfrac{\pi}{\sin(-\beta\pi)\Gamma(1+\mu_j)\Gamma(-\mu_j-\al-\beta)}.
\end{aligned}
\end{equation}
Hence, the solutions at the endpoint $x=1$ have the following decompositions for $k\in\NN$ such that $k\leq n$:
\begin{align*}
    f_j^{[2n-k]}(1)&=c_jf_j^{[2n-k]}(-1)-\gamma_jg_j^{[2n-k]}(-1) \\
    &=\left.\begin{cases}
    0 & \text{ if }k\neq j \\
    \gamma_k & \text{ if } k=j
    \end{cases}\right\},
\end{align*}
\begin{align*}
    g_j^{[2n-k]}(1)&=c_jf_j^{[2n-k]}(-1)-\eps_jg_j^{[2n-k]}(-1) \\
    &=\left.\begin{cases}
    0 & \text{ if }k\neq j \\
    \eps_k & \text{ if } k=j
    \end{cases}\right\},
\end{align*}
due to the initial conditions, for some constants $c_j$. The other operations act on the solutions similarly:
\begin{align*}
    f_j^{\{k-1\}}(1)&=\delta_jf_j^{\{k-1\}}(-1)+d_jg_j^{\{k-1\}}(-1) \\
    &=\left.\begin{cases}
    0 & \text{ if }k\neq j \\
    \delta_k & \text{ if } k=j
    \end{cases}\right\},
\end{align*}
\begin{align*}
    g_j^{\{k-1\}}(1)&=\eta_jf_j^{\{k-1\}}(-1)+d_jg_j^{\{k-1\}}(-1) \\
    &=\left.\begin{cases}
    0 & \text{ if }k\neq j \\
    \eta_k & \text{ if } k=j
    \end{cases}\right\},
\end{align*}
for some constants $d_j$. Note that the functions $f_j$ and $g_j$ are used in these formulas with the variable $x$, as the $2^{\beta}$ in the constants comes from switching the variable back from $t$.

We now utilize Definition \ref{d-mfunction} to determine the $m$-function associated with the self-adjoint extension ${\bf A}_0^n$ corresponding to $\Gamma_0$, which is the restriction of $\cD\ti{max}^{{\bf J},n}$ defined on 
\begin{align}\label{e-dirichletdomain}
\dom {\bf A}_0^n:=\left\{f\in\cD\ti{max}^{{\bf J},n} ~:~ f\in\ker(\Gamma_0)\right\}.
\end{align}
The boundary condition for the domain is an analog of the Neumann boundary conditions applied to regular Sturm--Liouville differential operators, as we will see.

Finally, we denote by $\fI_m$ the $m\times m$ square matrix with $1$'s on the anti-diagonal and $\cI_m$ the identity matrix of size $m$. Notice that, thanks to equations \eqref{e-nfirstconnect} and \eqref{e-nsecondconnect}, $\Gamma_0\{f_j,g_j\}$ is equal to
\begin{equation*}
\sbox0{$\begin{matrix} & & \\  & 0 &  \\  & & \end{matrix}$}
\sbox1{$\begin{matrix} & &  \\  & \fI_n &  \\  &  & \end{matrix}$}
\sbox2{$\begin{matrix} f_1^{[n]}(1) & \dots & f_n^{[n]}(1) \\ \vdots & \iddots & \vdots \\ f_1^{[2n-1]}(1) & \dots & f_n^{[2n-1]}(1) \end{matrix}$}
\sbox3{$\begin{matrix} g_1^{[n]}(1) & \dots & g_n^{[n]}(1) \\ \vdots & \iddots & \vdots \\ g_1^{[2n-1]}(1) & \dots & g_n^{[2n-1]}(1)\end{matrix}$}
\left(
\begin{array}{c|c}
\usebox{0}&\usebox{1}\vspace{-.35cm}\\\\
\hline\vspace{-.35cm}\\
 \vphantom{\usebox{0}}\usebox{2}&\usebox{3}
\end{array}
\right),
\end{equation*}
and simplifies to
\begin{equation}\label{e-gamma0plugin}
\Gamma_0\{f_j,g_j\}=
\sbox0{$\begin{matrix} & & \\  & 0 &  \\  & &\end{matrix}$}
\sbox1{$\begin{matrix} &  &  \\  & \fI_n &  \\  &  & \end{matrix}$}
\sbox2{$\begin{matrix}0 &  & \gamma_n \\  & \iddots &  \\ \gamma_1 &  & 0\end{matrix}$}
\sbox3{$\begin{matrix} 0 &  & \eps_n \\  & \iddots &  \\ \eps_1 &  & 0\end{matrix}$}
\left(
\begin{array}{c|c}
\usebox{0}&\usebox{1}\vspace{-.35cm}\\\\
\hline\vspace{-.35cm}\\
 \vphantom{\usebox{0}}\usebox{2}&\usebox{3}
\end{array}
\right),
\end{equation}
from the connecting formulas above. The analogous expression for $\Gamma_1\{f_j,g_j\}$ similarly collapses to
\begin{equation}\label{e-gamma1plugin}
\Gamma_1\{f_j,g_j\}=
\sbox0{$\begin{matrix} & & \\  & \fI_n &  \\  & &\end{matrix}$}
\sbox1{$\begin{matrix} &  &  \\  & 0 &  \\  &  & \end{matrix}$}
\sbox2{$\begin{matrix} 0 & & \delta_n \\  & \iddots &  \\ \delta_1 &  & 0\end{matrix}$}
\sbox3{$\begin{matrix} 0 &  & \eta_n \\  & \iddots &  \\ \eta_1 &  & 0\end{matrix}$}
\left(
\begin{array}{c|c}
\usebox{0}&\usebox{1}\vspace{-.35cm}\\\\
\hline\vspace{-.35cm}\\
 \vphantom{\usebox{0}}\usebox{2}&\usebox{3}
\end{array}
\right).
\end{equation}
Recall the well-known block matrix inversion formula
\begin{equation}
    \left(\begin{array}{cc}
    A & B \\
    C & D 
    \end{array}\right)
    =\left(\begin{array}{cc}
    (A-BD^{-1}C)^{-1} & -(A-BD^{-1}C)^{-1}BD^{-1} \\
    -D^{-1}C(A-BD^{-1}C)^{-1} & D^{-1}(I+C(A-BD^{-1}C)^{-1}BD^{-1} 
    \end{array}\right),
\end{equation}
which is valid when $D$ and $A-BD^{-1}C$ are invertible. Note that $\Gamma_0\{f_j,g_j\}$ satisfies these properties (matching the quadrants appropriately) as $D$ is trivially invertible and 
\begin{equation*}
    (A-BD^{-1}C)^{-1}=\left(\begin{array}{ccc}
    0 & & -\dfrac{\eps_1}{\gamma_1} \\
    & \iddots & \\
    -\dfrac{\eps_n}{\gamma_n} & & 0
    \end{array}\right).
\end{equation*}
Definition \ref{d-mfunction} thus yields
\begin{align}\label{e-0mfunction}
        \Gamma_1\left(\Gamma_0\{f_j,g_j\}\right)^{-1}&=
\sbox0{$\begin{matrix} & & \\  & \fI_n &  \\  & &\end{matrix}$}
\sbox1{$\begin{matrix} &  &  \\  & 0 &  \\  &  & \end{matrix}$}
\sbox2{$\begin{matrix} 0 & & \delta_n \\  & \iddots &  \\ \delta_1 &  & 0\end{matrix}$}
\sbox3{$\begin{matrix} 0 &  & \eta_n \\  & \iddots &  \\ \eta_1 &  & 0\end{matrix}$}
\left(
\begin{array}{c|c}
\usebox{0}&\usebox{1}\vspace{-.35cm}\\\\
\hline\vspace{-.35cm}\\
 \vphantom{\usebox{0}}\usebox{2}&\usebox{3}
\end{array}
\right)
\sbox0{$\begin{matrix} 0 & & -\dfrac{\eps_1}{\gamma_1}\\  & \iddots &  \\  -\dfrac{\eps_n}{\gamma_n} & & 0\end{matrix}$}
\sbox1{$\begin{matrix} 0 &  & \dfrac{1}{\gamma_1} \\  & \iddots &  \\ \dfrac{1}{\gamma_n} &  & 0 \end{matrix}$}
\sbox2{$\begin{matrix}  & &  \\  & \fI_n &  \\   &  & \end{matrix}$}
\sbox3{$\begin{matrix}  &  &  \\  & 0 &  \\  &  &  \end{matrix}$}
\left(
\begin{array}{c|c}
\usebox{0}&\usebox{1}\vspace{-.35cm}\\\\
\hline\vspace{-.35cm}\\
 \vphantom{\usebox{0}}\usebox{2}&\usebox{3}
\end{array}
\right)\nonumber \\
&=
\sbox0{$\begin{matrix}  -\dfrac{\eps_n}{\gamma_n} & & 0 \\  & \ddots &  \\  0 &  & -\dfrac{\eps_1}{\gamma_1} \end{matrix}$}
\sbox1{$\begin{matrix} \dfrac{1}{\gamma_n} &  & 0 \\  & \ddots &  \\ 0 &  & \dfrac{1}{\gamma_1} \end{matrix}$}
\sbox2{$\begin{matrix} \dfrac{1}{\gamma_n} & & 0 \\  & \ddots &  \\ 0 &  &  \dfrac{1}{\gamma_1} \end{matrix}$}
\sbox3{$\begin{matrix} \dfrac{\delta_n}{\gamma_n} &  & 0 \\  & \ddots &  \\ 0 &  & \dfrac{\delta_1}{\gamma_1} \end{matrix}$}
\left(
\begin{array}{c|c}
\usebox{0}&\usebox{1}\vspace{-.35cm}\\\\
\hline\vspace{-.35cm}\\
 \vphantom{\usebox{0}}\usebox{2}&\usebox{3}
\end{array}
\right):=M_0^n(\la).
\end{align}
The lower left quadrant in the last matrix was simplified by the assumption that $[f_j,g_j](1)=1$ for all $j$. Indeed, the entries for this quadrant are of the following form:
\begin{align}\label{e-sesquisimplify}
    \dfrac{\gamma_j\eta_j-\delta_j\eps_j}{\gamma_j}&=\dfrac{f_j^{[2n-j]}(1)g_j^{\{j-1\}}(1)-f_j^{\{j-1\}}(1)g_j^{[2n-j]}(1)}{\gamma_j} \\
    &=\dfrac{[f_j,g_j](1)}{\gamma_j}=\dfrac{1}{\gamma_j}. \nonumber
\end{align}

The constant $e_j$ is important here; if it is omitted from the definition of $\gamma_j$ and $\delta_j$, the expression $\gamma_j\eta_j-\delta_j\eps_j=[\widetilde{f}_j,g_j](1)$ still holds and can be simplified using
\begin{align*}
    \Gamma(1-z)\Gamma(z)=\dfrac{\pi}{\sin(\pi z)} \text{ for }z\notin\ZZ,
\end{align*}
and some trigonometric identities to determine
\begin{align}\label{e-edefn}
e_j=\dfrac{\sin(\pi\beta)}{2^{\beta}\sin(\pi\al)},
\end{align}
when $\mu_j+\al+1$, $\mu_j+\beta+1$, $1+\mu_j$, and $\mu_j+\al+\beta+1$ are not in $\ZZ$. The surprising part is that in these cases $e_j$ does not depend on the spectral parameter, and these conditions are associated with the spectrum of important self-adjoint extensions, as will be shown shortly. 

$M_0^n(\la)$ is the $m$-function for the boundary triple $\{\CC^{2n},\Gamma_0,\Gamma_1\}$ associated with the self-adjoint extension ${\bf A}_0^n$ of $\cD\ti{min}^{{\bf J},n}$ and domain given in equation \eqref{e-dirichletdomain}. The spectrum of ${\bf A}_0^n$ is discrete and eigenvalues are located at those $z\in\CC$ which are poles of $M_0^n(\la)$. 

The function $1/\Gamma(z)$ is analytic, so all spectral information will come from the $1/\gamma_j$ terms. Poles thus occur when $-\mu_j=-m$ or $\mu_j+\al+\beta+1=-m$, for $m\in\NN_0$. Hence, $\mu_j=m$ or $\mu_j=-m-\al-\beta-1$. In both cases, the formula $\la=\mu_j^n(\mu_j+\al+\beta+1)^n$ yields the same result and the spectrum, $\sigma_0^n$, of ${\bf A}_0^n$ is
\begin{align}\label{e-a0spectrum}
    \sigma_0^n=\left\{ m^n(m+\al+\beta+1)^n ~:~ m\in\NN_0 \right\}.
\end{align}

The operator ${\bf A}_0$ clearly includes all of the Jacobi polynomials in its domain, and upon inspection they must be the eigenfunctions associated with the eigenvalues in equation \eqref{e-a0spectrum}. We conclude that $\dom {\bf A}_0$ coincides with the $n/2$ left-definite domain due to \cite[Theorem 5.2]{FL}, which analyzes such domains. See also \cite{EKLWY, LW02} for more on left-definite operators and domains. It seems that this is the first instance in which it can be stated that a left-definite domain is actually the Friedrichs extension.

\begin{cor}
The operator ${\bf A}^n_0$ associated with $\dom {\bf A}_0^n$ is the Friedrichs extension of the minimal operator ${\bf A}\ti{min}^n$ associated with $\cD\ti{min}^{{\bf J},n}$.
\end{cor}

\begin{proof}
Let the Friedrichs Extension associated with the differential expression $\ell^n_{{\bf J}}$ be denoted by ${\bf A}_F$. Theorem 13 of \cite{MZ} states that the boundary conditions imposed on $\cD\ti{max}^{{\bf J},n}$ are of the form 
\begin{align*}
    \dom{\bf A}_F=\left\{f\in\cD\ti{max}^{{\bf J},n} ~:~[f,y^{(k)}]_n(-1)=[f,y^{(n+k)}]_n(1)=0,~k=1,\dots,n.\right\},
\end{align*}
where $y^{(k)}$ and $y^{(n+k)}$ denote the principal solutions of the equation $\ell_{\bf J}^n[f]=\la f$ at the endpoints $x=-1$ and $x=1$ respectively. Notice that if $f\in\dom{\bf A}_0^n$ then $f\in\dom{\bf A}_F$ also, as the principal solutions must take the same form near the relevant endpoint as $f_j$ from equation \eqref{e-nsolution} and the decomposition from Corollary \ref{c-newerdefectspaces} still holds. Hence, $\dom{\bf A}_0^n\subseteq\dom{\bf A}_F$. However, both domains are extensions of $\cD\ti{min}^{{\bf J},n}$ by $2n$ dimensions and are associated with self-adjoint operators, so we conclude that $\dom{\bf A}_0^n=\dom{\bf A}_F$.
\end{proof}

The above process can be repeated to determine the $m$-function associated with the self-adjoint extension ${\bf A}_1^n$ corresponding to $\Gamma_1$, which is the restriction of $\cD\ti{max}^{{\bf J},n}$ defined on 
\begin{align}\label{e-neumanndomain}
\dom {\bf A}_1^n:=\left\{f\in\cD\ti{max}^{{\bf J},n} ~:~ f\in\ker(\Gamma_1)\right\}.
\end{align}
The boundary condition for the domain is an analog of the Dirichlet boundary conditions applied to regular Sturm--Liouville differential operators, as we will see.

The formulas \eqref{e-gamma0plugin} and \eqref{e-gamma1plugin} can again be used in conjunction with Definition \ref{d-mfunction} to determine $\Gamma_0\left(\Gamma_1\{f_j,g_j\}\right)^{-1}$, but first $(\Gamma_1\{f_j,g_j\}^{-1}$ must be computed. Again, matrix block inversion can be used, this time with the alternate formula
\begin{align*}
    \left(\begin{array}{cc}
    A & B \\
    C & D 
    \end{array}\right)
    =\left(\begin{array}{cc}
    A^{-1}(I+B(D-CA^{-1}B)^{-1}CA^{-1}) & -AB(D-CA^{-1}B)^{-1} \\
    -(D-CA^{-1}B)^{-1}CA^{-1} & (D-CA^{-1}B)^{-1} 
    \end{array}\right),
\end{align*}
which is valid when $A$ and $D-CA^{-1}B$ are invertible. Note that $\Gamma_1\{f_j,g_j\}$ satisfies these properties (matching the quadrants appropriately) as $A$ is trivially invertible and 
\begin{equation*}
    (D-CA^{-1}B)^{-1}=\left(\begin{array}{ccc}
    0 & & \dfrac{1}{\eta_1} \\
    & \iddots & \\
    \dfrac{1}{\eta_n} & & 0
    \end{array}\right).
\end{equation*}

Thus, calculate
\begin{align}\label{e-1mfunction}
        \Gamma_0\left(\Gamma_1\{f_j,g_j\}\right)^{-1}&=
\sbox0{$\begin{matrix} & & \\  & 0 &  \\  & &\end{matrix}$}
\sbox1{$\begin{matrix} &  &  \\  & \fI_n &  \\  &  & \end{matrix}$}
\sbox2{$\begin{matrix} 0 & & \gamma_n \\  & \iddots &  \\ \gamma_1 &  & 0\end{matrix}$}
\sbox3{$\begin{matrix} 0 &  & \eps_n \\  & \iddots &  \\ \eps_1 &  & 0\end{matrix}$}
\left(
\begin{array}{c|c}
\usebox{0}&\usebox{1}\vspace{-.35cm}\\\\
\hline\vspace{-.35cm}\\
 \vphantom{\usebox{0}}\usebox{2}&\usebox{3}
\end{array}
\right)
\sbox0{$\begin{matrix}  & & \\  & \fI_n &  \\  &  &  \end{matrix}$}
\sbox1{$\begin{matrix} &  &  \\  & 0 &  \\  &  & \end{matrix}$}
\sbox2{$\begin{matrix} 0 & & -\dfrac{\delta_1}{\eta_1} \\  & \iddots &  \\ -\dfrac{\delta_n}{\eta_n}  &  & 0 \end{matrix}$}
\sbox3{$\begin{matrix} 0 &  & \dfrac{1}{\eta_1} \\  & \iddots &  \\  \dfrac{1}{\eta_n} &  & 0 \end{matrix}$}
\left(
\begin{array}{c|c}
\usebox{0}&\usebox{1}\vspace{-.35cm}\\\\
\hline\vspace{-.35cm}\\
 \vphantom{\usebox{0}}\usebox{2}&\usebox{3}
\end{array}
\right) \nonumber\\
&=
\sbox0{$\begin{matrix}  -\dfrac{\delta_n}{\eta_n} & & 0 \\  & \ddots &  \\  0 &  & -\dfrac{\delta_1}{\eta_1} \end{matrix}$}
\sbox1{$\begin{matrix} \dfrac{1}{\eta_n} &  & 0 \\  & \ddots &  \\ 0 &  & \dfrac{1}{\eta_1} \end{matrix}$}
\sbox2{$\begin{matrix} \dfrac{1}{\eta_n} & & 0 \\  & \ddots &  \\ 0 &  &  \dfrac{1}{\eta_1} \end{matrix}$}
\sbox3{$\begin{matrix} \dfrac{\eps_n}{\eta_n} &  & 0 \\  & \ddots &  \\ 0 &  & \dfrac{\eps_1}{\eta_1} \end{matrix}$}
\left(
\begin{array}{c|c}
\usebox{0}&\usebox{1}\vspace{-.35cm}\\\\
\hline\vspace{-.35cm}\\
 \vphantom{\usebox{0}}\usebox{2}&\usebox{3}
\end{array}
\right):=M_1^n(\la).
\end{align}
The entries of the lower left quadrant of $M_1^n(\la)$ were again simplified by equation \eqref{e-sesquisimplify}. $M_1^n(\la)$ is the $m$-function for the boundary triple $\{\CC^{2n},\Gamma_0,\Gamma_1\}$ associated with the self-adjoint extension ${\bf A}_1^n$ of $\cD\ti{min}^{{\bf J},n}$ and domain given in equation \eqref{e-neumanndomain}. The spectrum of ${\bf A}_1^n$ is discrete and eigenvalues are located at those $z\in\CC$ that are poles of $M_1^n(\la)$, which come from the $1/\eta_j$ terms. 

Poles thus occur when $1+\mu_j=-m$ or $-\mu_j-\al-\beta=-m$, for $m\in\NN_0$. Hence, $\mu_j=-m-1$ or $\mu_j=m-\al-\beta$. In both cases, the formula $\la=\mu_j^n(\mu_j+\al+\beta+1)^n$ yields the same result and the spectrum, $\sigma_1^n$, of ${\bf A}_1^n$ is
\begin{align}\label{e-a1spectrum}
    \sigma_1^n=\left\{ (m+1)^n(m-\al-\beta)^n ~:~ m\in\NN_0 \right\}.
\end{align}

The operator ${\bf A}_1$ clearly includes all of the Jacobi functions of the second kind in its domain, and upon inspection the eigenfunctions associated with the eigenvalues in equation \eqref{e-a1spectrum} must be such functions. 

It is thought that the operator ${\bf A}_1$ is the von Neumann--Krein extension of $\cD\ti{min}^{{\bf J},n}$. Unfortunately, a theorem analogous to \cite[Theorem 13]{MZ} that relates this extension with the non-principal solutions does not seem to exist in the literature.

%%%%%%%%%%%%%%%%%%%%%%%%%%%%%
%%%%%%%%%%%%%%%%%%%%%%%%%%%%%
\subsection{Other Boundary Conditions}\label{ss-otherbc}
%%%%%%%%%%%%%%%%%%%%%%%%%%%%%
%%%%%%%%%%%%%%%%%%%%%%%%%%%%%

The two-self-adjoint extensions analyzed thus far were naturally defined by the choice of the boundary triple. But the theory of boundary triples provides many tools for analyzing other boundary conditions. Two additional scenarios are now explored: when so-called separated boundary conditions are imposed, and when an analog of periodic boundary conditions are imposed. In each case, an explicit $m$-function is able to be derived thanks to the formula for $M_0^n(\la)$ in equation \eqref{e-0mfunction}.

Consider the matrix 
\begin{align*}
\te=
    \left(\begin{array}{ccc}
    c_1 & & 0 \\
     & \ddots & \\
     0 & & c_{2n}
    \end{array}\right),
\end{align*}
so that the operator ${\bf A}_{\te}^n$ is the self-adjoint extension of $\cD\ti{min}^{{\bf J},n}$ which acts on
\begin{align}\label{e-newthetadomain}
    \dom {\bf A}_{\te}^n=\left\{f\in\cD\ti{max}^{{\bf J},n} ~:~ \te\Gamma_0[f]=\Gamma_1[f]\right\}.
\end{align}
There are thus $n$ boundary conditions imposed on $\dom {\bf A}_{\te}^n$ at each endpoint, each condition involving two quasi-derivatives: $-c_jf^{[n+j-1]}(-1)=f^{\{n-j\}}(-1)$ when $j\leq n$ and $-c_jf^{[j-1]}(1)=f^{\{2n-j\}}(1)$ when $n<j\leq 2n$.

The $m$-function associated with the self-adjoint extension ${\bf A}_{\te}^n$ is then given via equation \eqref{e-thetam} as $(\te-M_0^n(\la))^{-1}$, which can be simplified to
\begin{align}\label{e-nthetam}
M_{\te}^n(\la)=
\left(\begin{array}{cc}
     \widetilde{A} & \widetilde{B} \\
     \widetilde{C} & \widetilde{D}
\end{array}\right),
\end{align}
where
\begin{align*}
        \widetilde{A}&=\left(\begin{array}{ccc} \dfrac{\gamma_n(c_{n+1}\gamma_n-\delta_n)}{(c_1\gamma_n+\eps_n)(c_{n+1}\gamma_n-\delta_n)-1} & & 0 \\  & \ddots &  \\  0 &  & \dfrac{\gamma_1(c_{2n}\gamma_1-\delta_1)}{(c_n\gamma_1+\eps_1)(c_{2n}\gamma_1-\delta_1)-1}
    \end{array}\right), \\
    \widetilde{B}=\widetilde{C}&=\left(\begin{array}{ccc}
         \dfrac{-\gamma_n}{(c_1\gamma_n+\eps_n)(c_{n+1}\gamma_n-\delta_n)-1} & & 0 \\  & \ddots &  \\  0 &  & \dfrac{-\gamma_1}{(c_n\gamma_1+\eps_1)(c_{2n}\gamma_1-\delta_1)-1} 
    \end{array}\right), \\
    \widetilde{D}&=\left(\begin{array}{ccc}
       \dfrac{\gamma_n(c_1\gamma_n+\eps_n)}{(c_1\gamma_n+\eps_n)(c_{n+1}\gamma_n-\delta_n)-1} & & 0 \\  & \ddots &  \\  0 &  & \dfrac{\gamma_1(c_n\gamma_1+\eps_1)}{(c_n\gamma_1+\eps_1)(c_{2n}\gamma_1-\delta_1)-1}
    \end{array}\right).
\end{align*}

The spectral properties given by $M_{\te}^n(\la)$ are not as easy to determine as in the previous examples. They arise when
\begin{align*}
    (c_j\gamma_{n-j+1}+\eps_{n-j+1})(c_{n+j}\gamma_{n-j+1}-\delta_{n-j+1})=1,
\end{align*}
but the explicit values that $\mu_j$ must take are unclear. Indeed, closed form solutions may not be able to be determined for such an equation. The expression does resemble the form for regular Sturm--Liouville ($n=1$) operators with separated boundary conditions, i.e. \cite[Example 6.3.6]{BdS}. 

Finally, we consider an analog of periodic boundary conditions by using Theorem \ref{t-weyltransform}. Let
\begin{align}\label{e-w}
\cW=\dfrac{1}{\sqrt{2}}
\sbox0{$\begin{matrix}  \cI_n & \cI_n \\ 0 & 0 \end{matrix}$}
\sbox1{$\begin{matrix} 0 & 0 \\  \cI_n & -\cI_n \end{matrix}$}
\sbox2{$\begin{matrix} 0 & 0 \\  -\cI_n & \cI_n \end{matrix}$}
\sbox3{$\begin{matrix} \cI_n & \cI_n \\ 0 & 0 \end{matrix}$}
\left(
\begin{array}{c|c}
\usebox{0}&\usebox{1}\vspace{-.35cm}\\\\
\hline\vspace{-.35cm}\\
 \vphantom{\usebox{0}}\usebox{2}&\usebox{3}
\end{array}
\right)
=\left(\begin{array}{cc}
\cB^* & -\cA^* \\
\cA^* & \cB^*
\end{array}\right).
\end{align}
Note the $2n\times 2n$ matrices $\cA$ and $\cB$ satisfy the conditions $\cA^*\cB=\cB^*\cA$, $\cA\cB^*=\cB\cA^*$ and $\cA\cA^*+\cB\cB^*=I=\cA^*\cA+\cB^*\cB$. A new boundary triple for $\cD\ti{max}^{{\bf J},n}$ is thus $\{\CC^{2n},\Gamma_0',\Gamma_1'\}$, where $\Gamma_0'$ and $\Gamma_1'$ are given by 
\begin{align*}
    \left(\begin{array}{c}
    \Gamma_0' \\
    \Gamma_1'
    \end{array}\right)
    =\cW
    \left(\begin{array}{c}
    \Gamma_0 \\
    \Gamma_1
    \end{array}\right).
\end{align*}
Explicitly, the maps $\Gamma_0',\Gamma_1':\cD\ti{max}^{{\bf J},n}\to \CC^{2n}$ act via
\begin{align*}
\Gamma_0' f:=\left(
\begin{array}{c}
-f^{[n]}(-1)+f^{[n]}(1) \\
\vdots \\
-f^{[2n-1]}(-1)+f^{[2n-1]}(1) \\
f^{\{n-1\}}(-1)-f^{\{n-1\}}(1) \\
\vdots \\
f^{\{0\}}(-1)-f^{\{0\}}(1)
\end{array} 
\right), \hspace{.2cm}
\Gamma_1' f:=\left(
\begin{array}{c}
f^{\{n-1\}}(-1)+f^{\{n-1\}}(1) \\
\vdots \\
f^{\{0\}}(-1)+f^{\{0\}}(1) \\
f^{[n]}(-1)+f^{[n]}(1) \\
\vdots \\
f^{[2n-1]}(-1)+f^{[2n-1]}(1)
\end{array} 
\right),
\end{align*}
and there is a $1/\sqrt{2}$ factor hidden from each term for simplicity. The $m$-function associated with the self-adjoint extension of $\cD\ti{min}^{{\bf J},n}$ which acts on
\begin{align}\label{e-newwdomain}
    \dom {\bf B}_0^n=\left\{f\in\cD\ti{max}^{{\bf J},n} ~:~ f\in\ker\Gamma_0'\right\},
\end{align}
is then given by 
\begin{align*}
    M_0'(\la)=(\cA^*+\cB^*M_0^n(\la))(\cB^*-\cA^*M_0^n(\la))^{-1}.
\end{align*}
Of course this domain will only include functions that have the same value for each quasi-derivative at each endpoint. We have
\begin{align*}
M_0'(\la)=
\sbox0{$\begin{matrix}  \dfrac{2\eta_n}{\eps_n-\delta_n+2} & & 0 \\  & \ddots &  \\  0 &  & \dfrac{2\eta_1}{\eps_1-\delta_1+2} \end{matrix}$}
\sbox1{$\begin{matrix} \dfrac{\eps_n+\delta_n}{\eps_n-\delta_n+2} & & 0 \\  & \ddots &  \\  0 &  & \dfrac{\eps_1+\delta_1}{\eps_1-\delta_1+2} \end{matrix}$}
\sbox2{$\begin{matrix} \dfrac{\eps_n+\delta_n}{\eps_n-\delta_n+2} & & 0 \\  & \ddots &  \\  0 &  & \dfrac{\eps_1+\delta_1}{\eps_1-\delta_1+2} \end{matrix}$}
\sbox3{$\begin{matrix} \dfrac{2\gamma_n}{\eps_n-\delta_n+2} & & 0 \\  & \ddots &  \\  0 &  & \dfrac{2\gamma_1}{\eps_1-\delta_1+2} \end{matrix}$}
\left(
\begin{array}{c|c}
\usebox{0}&\usebox{1}\vspace{-.35cm}\\\\
\hline\vspace{-.35cm}\\
 \vphantom{\usebox{0}}\usebox{2}&\usebox{3}
\end{array}
\right).
\end{align*}
Note that in the calculation for the upper-left quadrant the fact that $(1+\eps_j\delta_j)/\gamma_j=\eta_j$ was used. The spectral properties revealed by $M_0'(\la)$ are again difficult to determine. Eigenvalues arise when
\begin{align*}
    \delta_j-\eps_j=2,
\end{align*}
but the explicit values that $\mu_j$ must take are unclear. The expression does resemble that arising from a simpler example for regular Sturm--Liouville ($n=1$) operators \cite[Example 6.3.6]{BdS}.

%%%%%%%%%%%%%%%%%%%%%%%%%%%%%
%%%%%%%%%%%%%%%%%%%%%%%%%%%%%
\section*{Acknowledgements}
%%%%%%%%%%%%%%%%%%%%%%%%%%%%%
%%%%%%%%%%%%%%%%%%%%%%%%%%%%%

The author would like to thank Annemarie Luger for helpful discussions at all stages of the project,  Fritz Gesztesy for assistance with references on multiple occasions and Jussi Behrndt for access to the book manuscript \cite{BdS}. He is also grateful to Constanze Liaw, Fritz Gesztesy, Lance Littlejohn and Roger Nichols for their correspondence.

%%%%%%%%%%%%%%%%%%%%%%%%%%%%%
%%%%%%%%%%%%%%%%%%%%%%%%%%%%%

\end{document}